\documentclass{amsart}
\usepackage{amssymb}
\usepackage{amsmath}
\usepackage{amsfonts}

\newtheorem{theorem}{Theorem}[section]

\newtheorem{conjecture}[theorem]{Conjecture}
\newtheorem{corollary}[theorem]{Corollary}

\newtheorem{definition}[theorem]{Definition}
\newtheorem{example}[theorem]{Example}

\newtheorem{lemma}[theorem]{Lemma}

\newtheorem{proposition}[theorem]{Proposition}
\newtheorem{remark}[theorem]{Remark}

\input{tcilatex}

\begin{document}
\title[Precosheaves are smooth]{Precosheaves of pro-sets and abelian
pro-groups are smooth}
\author{Andrei V. Prasolov}
\address{Institute of Mathematics and Statistics\\
University of Troms\o , N-9037 Troms\o , Norway}
\email{andrei.prasolov@uit.no}
\urladdr{http://www.math.uit.no/users/andreip/Welcome.html}
\date{}

\begin{abstract}
Let $\mathbb{D}$ be the category of pro-sets (or abelian pro-groups). It is
proved that for any Grothendieck site $X$, there exists a reflector from the
category of precosheaves on $X$ with values in $\mathbb{D}$ to the full
subcategory of cosheaves. In the case of precosheaves on topological spaces,
it is proved that any precosheaf is smooth, i.e. is locally isomorphic to a
cosheaf. Constant cosheaves are constructed, and there are established
connections with shape theory.
\end{abstract}

\maketitle
\tableofcontents

\setcounter{section}{-1}

\section{Introduction}

A \emph{presheaf} (\emph{precosheaf}) on a topological space $X$ with values
in a category $\mathbb{D}$ is just a contravariant (covariant) functor from
the category of open subsets of $X$ to $\mathbb{D}$, while a \emph{sheaf} (%
\emph{cosheaf}) is such a functor satisfying some extra conditions.
Therefore, a (pre)cosheaf with values in $\mathbb{D}$ is nothing but a
(pre)sheaf with values in the opposite category $\mathbb{D}^{op}$.

While the theory of sheaves is well developed, and is covered by a plenty of
publications, the theory of cosheaves is represented much poorer. The main
reason for this is that \emph{cofiltrant limits} are \emph{not} exact in the
\textquotedblleft usual\textquotedblright\ categories like sets or abelian
groups. On the contrary, \emph{filtrant colimits} are exact, which allows to
construct rather rich theories of sheaves (of sets or of abelian groups). To
sum up, the categories $\mathbb{SET}$ of small sets and $\mathbb{AB}$ of
small abelian groups (Definition \ref{SET,AB}) are badly suited for cosheaf
theory. Dually, $\mathbb{SET}^{op}$ and $\mathbb{AB}^{op}$ are badly suited
for sheaf theory.

According to the classical definition (\cite{Bredon-Book-MR1481706},
Definition I.1.2), a sheaf of sets (abelian groups) is a local homeomorphism%
\begin{equation*}
\pi :\mathcal{A}\longrightarrow X
\end{equation*}%
(with a structure of an abelian group on each \emph{stalk} $\mathcal{A}%
_{x}=\pi ^{-1}\left( x\right) $). However, the latter definition is
equivalent to the following: a sheaf $\mathcal{A}$ is a presheaf satisfying
conditions (S1) and (S2) from \cite{Bredon-Book-MR1481706}, I.1.7. Those two
conditions, in turn, can be easily reformulated for an arbitrary category $%
\mathbb{D}$, which gives a \textquotedblleft modern\textquotedblright\
definition of a sheaf (Definition \ref{Sheaf in D}), while the condition
(S1) alone gives the definition of a \emph{separated} presheaf (\emph{%
monopresheaf} in the terminology of \cite{Bredon-Book-MR1481706}), see
Definition$\ $\ref{Separated-presheaf}. We are not aware of any definition
of a cosheaf analogous to the classical definition of a sheaf. The
\textquotedblleft modern\textquotedblright\ definition, however, can be
easily applied to cosheaves. In \cite{Bredon-MR0226631}, Section 1, and \cite%
{Bredon-Book-MR1481706}, Definitions V.1.1 and VI.1.1, there are given
definitions of a cosheaf and a \emph{coseparated} precosheaf (\emph{%
epiprecosheaf} in the terminology of \cite{Bredon-Book-MR1481706}). These
definitions are easily applied to a general category $\mathbb{D}$, see
Definition \ref{Cosheaf in D} for a cosheaf and Definition \ref%
{Coseparated-precosheaf} for a coseparated precosheaf.

\begin{remark}
Notice that the direct sum $\oplus $\ of abelian groups is a coproduct $%
\sqcup $\ in the category $\mathbb{AB}$. That is why the symbol $\oplus $\
in \cite{Bredon-Book-MR1481706} is replaced by $\sqcup $\ in our definition.
\end{remark}

\begin{remark}
While a sheaf of abelian groups is still a sheaf considered as a presheaf of
sets, a similar statement is not true for cosheaves of abelian groups,
pro-groups, sets and pro-sets. The reason is that the forgetting functors%
\begin{eqnarray*}
\mathbb{AB} &\longrightarrow &\mathbb{SET}, \\
Pro\left( \mathbb{AB}\right) &\longrightarrow &Pro\left( \mathbb{SET}\right)
,
\end{eqnarray*}%
do \textbf{not commute} with colimits (e.g. coproducts). See Corollary \ref%
{AB-cosheaf-is-not-SET-cosheaf}.
\end{remark}

The first step in building a suitable theory of cosheaves would be
constructing a \emph{cosheaf associated with a precosheaf}. It is impossible
for \textquotedblleft usual\textquotedblright\ cosheaves, because of the
mentioned drawbacks of categories $\mathbb{SET}$ and $\mathbb{AB}$. In \cite%
{Bredon-Book-MR1481706} and \cite{Bredon-MR0226631}, it is made an attempt
to avoid this difficulty by introducing the so-called \emph{smooth}
precosheaves (Definition \ref{Def-smooth}). It is not clear whether one has
enough smooth precosheaves for building a suitable theory of cosheaves (see
Examples\ \ref{Non-smooth-precosheaf}-\ref{H0-is-cosheaf}). In fact, Glen E.
Bredon back in 1968 was rather pessimistic on the issue. See \cite%
{Bredon-MR0226631}, p. 2: \textquotedblleft \emph{The most basic concept in
sheaf theory is that of a sheaf generated by a given presheaf. In
categorical terminology this is the concept of a reflector from presheaves
to sheaves. We believe that there is not much hope for the existence of a
reflector from precosheaves to cosheaves}\textquotedblright . It seems that
he was still pessimistic in 1997: Chapter VI \textquotedblleft Cosheaves and 
\v{C}ech homology\textquotedblright\ of his book \cite{Bredon-Book-MR1481706}
is almost identical to \cite{Bredon-MR0226631}.

On the contrary, our approach seems to have solved the problem. If one
allows (pre)cosheaves to take values in a larger category, then the desired
reflector can be constructed. How large should that category be? The first
candidates are the categories of functors $\mathbb{SET}^{\mathbb{SET}}$ and $%
\mathbb{AB}^{\mathbb{AB}}$ (actually, duals to these categories since we are
dealing with (pre)cosheaves) (see Section \ref{Appendix-Pro-category}). The
two categories are very nice with respect to limits, colimits, exactness
etc. The major drawback, however, is that they are \emph{huge} which means
that they are not $\mathfrak{U}$-categories (Definition \ref{U-category})
where $\mathfrak{U}$ is the \emph{universe} we are using in this paper. Let $%
\mathbb{D}$ be either $\mathbb{SET}$ or $\mathbb{AB}$. Our task is to find
an appropriate category $\mathbb{E}^{op}$ between $\mathbb{D}^{op}$ (which
is too small) and $\mathbb{D}^{\mathbb{D}}$ (which is too large):%
\begin{equation*}
\mathbb{D}^{op}\subseteq \mathbb{E}^{op}\subseteq \mathbb{D}^{\mathbb{D}}.
\end{equation*}

It follows from our considerations in this paper, that the best candidate
for such $\mathbb{E}$ is the pro-category $Pro\left( \mathbb{D}\right) $
(see Section \ref{Appendix-Pro-category}). Our reflector, which is naturally
called \textquotedblleft \emph{cosheafification}\textquotedblright , could
be built like this: take a precosheaf with values in $\mathbb{E}$, then the
corresponding presheaf with values in $\mathbb{E}^{op}$, \emph{sheafify} as
a presheaf with values in $\mathbb{D}^{\mathbb{D}}$, and try to return back
to the category $\mathbb{E}^{op}$. However, this is not that simple. To
check that the resulting sheaf takes values in $\mathbb{E}^{op}$, is too
complicated (see Remark \ref{Belongs-to-pro-D}). Moreover, such a method
would not cosheafify a precosheaf explicitly. That is why we build our
construction \textquotedblleft from below\textquotedblright , mirroring the
well-known two-step process of sheafification:%
\begin{equation*}
\mathcal{A}\longmapsto \mathcal{A}^{+}\longmapsto \mathcal{A}^{++}=\mathcal{A%
}^{\#}.
\end{equation*}%
We have succeeded because of the niceness of the category $Pro\left( \mathbb{%
D}\right) $. For \textquotedblleft usual\textquotedblright\ precosheaves
(with values in $\mathbb{D}$) this two-step process does not work.

\begin{remark}
An interesting attempt is made in \cite{Schneiders-MR885939} where the
author sketches a sheaf theory on topological spaces with values in an
\textquotedblleft elementary\textquotedblright\ category $\mathbb{E}^{op}$,
which is equivalent to a cosheaf theory with values in the category $\mathbb{%
E}$. He proposes a candidate for such a category: $\mathbb{E}$ is the
category $Pro_{\mathfrak{V}}\left( \mathfrak{U}\text{-}\mathbb{AB}\right) $
of $\mathfrak{V}$-indexed abelian pro-$\mathfrak{U}$-groups where $\mathfrak{%
U}$ and $\mathfrak{V}$ are \textbf{two} universes, and the latter is larger
than the former ($\mathfrak{U\in V}$). This seems to be too much! One
universe should be enough. Moreover, the \textquotedblleft
usual\textquotedblright\ pro-category $Pro\left( \mathbb{AB}\right) $ is too
small to be used in cosheaf theory in \emph{loc. cit.}: the category $\left(
Pro\left( \mathbb{AB}\right) \right) ^{op}$ is \textbf{not} elementary.
\end{remark}

The reflector we have constructed guarantees that our precosheaves are
always smooth (Corollary \ref{Corollary-smooth}). Moreover, in Theorem \ref%
{Our-cosheaves-vs-Bredon}, we give necessary and sufficient conditions for
smoothness of a precosheaf with values in an \textquotedblleft
old\textquotedblright\ category ($\mathbb{SET}$ or $\mathbb{AB}$):\ it is
smooth iff our reflector applied to that precosheaf produces a cosheaf which
takes values in that old category.

Another difficulty in cosheaf theory is the lack of suitable \emph{constant}
cosheaves. In \cite{Bredon-Book-MR1481706} and \cite{Bredon-MR0226631}, such
cosheaves are constructed only for locally connected spaces. See Examples %
\ref{p0-prime-not-cosheaf}-\ref{Converging-sequence}.

In this paper, we begin a systematic study of cosheaves on topological
spaces (as well as on general Grothendieck \emph{sites}, see Definition \ref%
{Grothendieck-site}) with values in $Pro\left( \mathbb{SET}\right) $ and $%
Pro\left( \mathbb{AB}\right) $. In Theorem \ref{Main-Pro-Set}, the cosheaf $%
\mathcal{A}_{\#}$ associated with a precosheaf $\mathcal{A}$ is constructed,
giving a pair of adjoint functors and a reflector from the category of
precosheaves to the category of cosheaves. It appears that on a topological
space such a cosheaf is \emph{locally isomorphic} to the original precosheaf
(Theorem \ref{Main-local-iso-pro-SET}), implying that any precosheaf is
smooth (Corollary \ref{Corollary-smooth}). In Theorem \ref{Main-constant},
constant cosheaves are constructed. It turns out that they are closely
connected to shape theory. Namely, the constant cosheaf $\left( S\right)
_{\#}$ with values in $Pro\left( \mathbb{SET}\right) $ is isomorphic to the
pro-homotopy cosheaf $S\times pro$-$\pi _{0}$ (in particular $\left( \mathbf{%
pt}\right) _{\#}%
\simeq%
pro$-$\pi _{0}$), while the constant cosheaf $\left( A\right) _{\#}$ with
values in $Pro\left( \mathbb{AB}\right) $ is isomorphic to the pro-homology
cosheaf $pro$-$H_{0}\left( \_,A\right) $.

In future papers, we are planning to develop homology of cosheaves, i.e. to
study projective and flabby cosheaves, projective and flabby resolutions,
and to construct the left satellites 
\begin{equation*}
H_{n}\left( X,\mathcal{A}\right) :=L_{n}\Gamma \left( X,\mathcal{A}\right)
\end{equation*}%
of the global sections functor 
\begin{equation*}
H_{0}\left( X,\mathcal{A}\right) :=\Gamma \left( X,\mathcal{A}\right) .
\end{equation*}%
It is expected that deeper connections to shape theory will be discovered,
as is stated in the two Conjectures below:

\begin{conjecture}
\label{Satellites-H}On the site $NORM\left( X\right) $ (Example \ref%
{Site-NORM}), the left satellites of $H_{0}$ are naturally isomorphic to the
pro-homology:%
\begin{equation*}
H_{n}\left( X,A_{\#}\right) =H_{n}\left( X,pro\text{-}H_{0}\left(
\_,A\right) \right) 
\simeq%
pro\text{-}H_{n}\left( X,A\right) .
\end{equation*}%
If $X$ is paracompact Hausdorff, the above isomorphisms exist also for the
standard site $OPEN\left( X\right) $ (Example \ref{Site-TOP}).
\end{conjecture}

\begin{conjecture}
\label{Satellites-Pi}On the site $NORM\left( X\right) $, the \textbf{%
non-abelian} left satellites of $H_{0}$ are naturally isomorphic to the
pro-homotopy:%
\begin{eqnarray*}
H_{n}\left( X,S_{\#}\right) &=&H_{n}\left( X,S\times pro\text{-}\pi
_{0}\right) 
\simeq%
S\times pro\text{-}\pi _{n}\left( X\right) , \\
H_{n}\left( X,\left( \mathbf{pt}\right) _{\#}\right) &=&H_{n}\left( X,pro%
\text{-}\pi _{0}\right) 
\simeq%
pro\text{-}\pi _{n}\left( X\right) .
\end{eqnarray*}%
If $X$ is paracompact Hausdorff, the above isomorphisms exist also for the
standard site $OPEN\left( X\right) $.
\end{conjecture}

The main application (Theorem \ref{Main-constant}) deals with the case of
topological spaces (i.e. the site $OPEN\left( X\right) $). Our
constructions, however, are valid for general Grothendieck sites. The
constructions in (strong) shape theory use essentially \emph{normal}
coverings instead of general coverings, therefore dealing with the site $%
NORM\left( X\right) $ instead of the site $OPEN\left( X\right) $. It seems
that Theorem \ref{Main-constant} is valid also for the site $NORM\left(
X\right) $. Applying our machinery (from this paper and from future papers)
to the site $FINITE\left( X\right) $ (Example \ref{Site-FINITE}), we expect
to obtain results on homology of the Stone-%
\u{C}ech \ %
compactification $\beta \left( X\right) $. To deal with the equivariant
homology, one should apply the machinery to the equivariant site $%
OPEN_{G}\left( X\right) $ (Example \ref{Equivariant}).

It is not yet clear how to generalize the above Conjectures to \emph{strong
shape theory}. However, we have some ideas how to do that.

Other possible applications could be in \'{e}tale homotopy theory \cite%
{Atrin-Mazur-MR883959} as is summarized in the following

\begin{conjecture}
Let $X^{et}$ be the site from Example \ref{Site-ETALE}.

\begin{enumerate}
\item The left satellites of $H_{0}$ are naturally isomorphic to the \'{e}%
tale pro-homology:%
\begin{equation*}
H_{n}\left( X^{et},A_{\#}\right) 
\simeq%
H_{n}^{et}\left( X,A\right) .
\end{equation*}

\item The non-abelian left satellites of $H_{0}$ are naturally isomorphic to
the \'{e}tale pro-homotopy:%
\begin{equation*}
H_{n}\left( X^{et},\left( \mathbf{pt}\right) _{\#}\right) 
\simeq%
\pi _{n}^{et}\left( X\right) .
\end{equation*}
\end{enumerate}
\end{conjecture}

\section{Main results}

Let $X$ be a site (Definition \ref{Grothendieck-site}), let $\mathbb{D}$ be
either $\mathbb{SET}$ or $\mathbb{AB}$, and let $\mathbb{PCS}\left(
X,Pro\left( \mathbb{D}\right) \right) $ and $\mathbb{CS}\left( X,Pro\left( 
\mathbb{D}\right) \right) $ be the categories of precosheaves and cosheaves,
respectively, on $X$, with values in $Pro\left( \mathbb{D}\right) $
(Definition \ref{Categories-of-(pre)cosheaves}). See Appendix (Section \ref%
{Appendix-Pro-category}) for the definition and properties of the categories 
$Pro\left( \mathbb{SET}\right) $ and $Pro\left( \mathbb{AB}\right) $.

\begin{theorem}
\label{Main-Pro-Set}Let $\mathbb{D}$ be either $\mathbb{SET}$ or $\mathbb{AB}
$.

\begin{enumerate}
\item The inclusion functor%
\begin{equation*}
I:\mathbb{CS}\left( X,Pro\left( \mathbb{D}\right) \right) \longrightarrow 
\mathbb{P}\mathbb{CS}\left( X,Pro\left( \mathbb{D}\right) \right)
\end{equation*}%
has a right adjoint%
\begin{equation*}
\left( {}\right) _{\#}:\mathbb{P}\mathbb{CS}\left( X,Pro\left( \mathbb{D}%
\right) \right) \longrightarrow \mathbb{CS}\left( X,Pro\left( \mathbb{D}%
\right) \right) .
\end{equation*}

\item For any cosheaf $\mathcal{A}$ on $X$ with values in $Pro\left( \mathbb{%
D}\right) $, the canonical morphism $\mathcal{A}_{\#}\rightarrow \mathcal{A}$
is an isomorphism of cosheaves, i.e. $\left( {}\right) _{\#}$ is a reflector
from the category of precosheaves with values in $Pro\left( \mathbb{D}%
\right) $ to the full subcategory of cosheaves with values in the same
category.
\end{enumerate}
\end{theorem}

In Theorems \ref{Main-local-iso-pro-SET} and \ref{Main-constant}, as well as
in Corollary \ref{Corollary-smooth} below, $X$ is a topological space. We
denote by the same letter $X$ the corresponding standard site (see Example %
\ref{Site-TOP} and Remark \ref{Denote-standard-site-simply}).

\begin{theorem}
\label{Main-local-iso-pro-SET}Let $\mathbb{D}$ be either $\mathbb{SET}$ or $%
\mathbb{AB}$.

\begin{enumerate}
\item For any precosheaf $\mathcal{A}$ on $X$ with values in $Pro\left( 
\mathbb{D}\right) $, $\mathcal{A}_{\#}\rightarrow \mathcal{A}$ is a local
isomorphism (Definition \ref{local-isomorphism}).

\item Any local isomorphism $\mathcal{A\rightarrow B}$ between cosheaves on $%
X$ with values in $Pro\left( \mathbb{D}\right) $, is an isomorphism.

\item If $\mathcal{B}\longrightarrow \mathcal{A}$ is a local isomorphism,
and $\mathcal{B}$ is a cosheaf, then the natural morphism $\mathcal{B}%
\rightarrow \mathcal{A}_{\#}$ is an isomorphism.
\end{enumerate}
\end{theorem}

Corollary \ref{Corollary-smooth} below guarantees that all precosheaves with
values in $Pro\left( \mathbb{SET}\right) $ or $Pro\left( \mathbb{AB}\right) $
are smooth:

\begin{definition}
\label{Def-smooth}A precosheaf $\mathcal{A}$ is called \textbf{smooth} (\cite%
{Bredon-Book-MR1481706}, Corollary VI.3.2 and Definition VI.3.4, or \cite%
{Bredon-MR0226631}, Corollary 3.5 and Definition 3.7) iff there exist
precosheaves $\mathcal{B}$ and $\mathcal{B}^{\prime }$, a cosheaf $\mathcal{C%
}$, and local isomorphisms $\mathcal{A\rightarrow B\leftarrow C}$, or,
equivalently, local isomorphisms $\mathcal{A\leftarrow B}^{\prime
}\rightarrow \mathcal{C}$.
\end{definition}

\begin{corollary}
\label{Corollary-smooth}Any precosheaf with values in $Pro\left( \mathbb{SET}%
\right) $ or in $Pro\left( \mathbb{AB}\right) $ is smooth.
\end{corollary}

\begin{proof}
$\mathcal{A}\overset{Id}{\rightarrow }\mathcal{A\leftarrow A}_{\#}$ or $%
\mathcal{A\leftarrow A}_{\#}\overset{Id}{\rightarrow }\mathcal{A}_{\#}$.
\end{proof}

The results on cosheaves and precosheaves with values in $Pro\left( \mathbb{%
SET}\right) $ and $Pro\left( \mathbb{AB}\right) $ can be applied to
\textquotedblleft usual\textquotedblright\ ones (like in \cite%
{Bredon-MR0226631} and \cite{Bredon-Book-MR1481706}, Chapter VI), with
values in $\mathbb{SET}$ and $\mathbb{AB}$. The thing is that $\mathbb{SET}$
is a full subcategory of $Pro\left( \mathbb{SET}\right) $, while $\mathbb{AB}
$ is a full abelian subcategory of $Pro\left( \mathbb{AB}\right) $.

The connection between the two types of (pre)cosheaves can be summarized in
the following

\begin{theorem}
\label{Our-cosheaves-vs-Bredon}Let $\mathbb{D}$ be either $\mathbb{SET}$ or $%
\mathbb{AB}$, and let $\mathcal{A}$ be a precosheaf on a topological space $%
X $ with values in $\mathbb{D}$.

\begin{enumerate}
\item $\mathcal{A}$ is coseparated (a cosheaf) iff it is coseparated (a
cosheaf) when considered as a precosheaf with values in $Pro\left( \mathbb{D}%
\right) $.

\item $\mathcal{A}$ is smooth iff $\mathcal{A}_{\#}$ takes values in $%
\mathbb{D}$, i.e. $\mathcal{A}_{\#}\left( U\right) \in \mathbb{D}$ (in other
words, is a \emph{trivial} pro-object, see Remark \ref{Trivial-pro-object})
for any open subset $U\subseteq X$.
\end{enumerate}
\end{theorem}

\begin{remark}
\label{Local-isomorphism-on-NORM(X)}Our reflector $\left( {}\right) _{\#}$
can be surely applied to precosheaves on the sites $NORM\left( X\right) $, $%
FINITE\left( X\right) $ and $OPEN_{G}\left( X\right) $. It is not clear,
however, whether a statement analogous to Theorem \ref%
{Main-local-iso-pro-SET}, is true. It seems that it may be true for the site 
$NORM\left( X\right) $ if one modifies the definition of a costalk, taking 
\textbf{normal} instead of arbitrary neighborhoods.
\end{remark}

We are now able to construct \emph{constant} cosheaves, and to establish
connections to shape theory.

\begin{theorem}
\label{Main-constant}Let $S$ be a set, and let $A$ be an abelian group.

\begin{enumerate}
\item The precosheaf%
\begin{equation*}
\mathcal{P}\left( U\right) :=S\times pro\text{-}\pi _{0}\left( U\right)
\end{equation*}%
where $pro$-$\pi _{0}$ is the pro-homotopy functor from Definition \ref%
{Pro-homotopy-groups} (see also \cite{Mardesic-Segal-MR676973}, p. 121), is
a cosheaf.

\item Let $S$ be the constant precosheaf corresponding to $S$ (Definition %
\ref{constant}) on $X$ with values in $Pro\left( \mathbb{SET}\right) $. Then 
$\left( S\right) _{\#}$ is naturally isomorphic to $\mathcal{P}$.

\item The precosheaf%
\begin{equation*}
\mathcal{H}\left( U\right) :=pro\text{-}H_{0}\left( U,A\right)
\end{equation*}%
where $pro$-$H_{0}$ is the pro-homology functor from Definition \ref%
{Pro-homology-groups} (see also \cite{Mardesic-Segal-MR676973}, , \S %
II.3.2), is a cosheaf.

\item Let $A$ be the constant precosheaf corresponding to $A$ (Definition %
\ref{constant-AB}) on $X$ with values in $Pro\left( \mathbb{AB}\right) $.
Then $\left( A\right) _{\#}$ is naturally isomorphic to $\mathcal{H}$.
\end{enumerate}
\end{theorem}

\begin{corollary}
\label{One-point-constant}
\end{corollary}

\begin{enumerate}
\item $pro$-$\pi _{0}$ is a cosheaf.

\item $\left( \mathbf{pt}\right) _{\#}%
\simeq%
pro$-$\pi _{0}$ where $\mathbf{pt}$ is the one-point constant precosheaf.
\end{enumerate}

\begin{proof}
Put $S=\mathbf{pt}$ in Theorem \ref{Main-constant} (1-2).
\end{proof}

\section{Cosheaves and precosheaves}

\begin{definition}
\label{Precosheaf in D}Let $X=\left( Cat\left( X\right) ,Cov\left( X\right)
\right) $ be a site. A \textbf{precosheaf} on $X$ \textbf{with values in} $%
\mathbb{D}$ is a \textbf{covariant} functor $\mathcal{A}$ from $Cat\left(
X\right) $ to $\mathbb{D}$. \textbf{Morphisms} between precosheaves are
morphisms between the corresponding functors.
\end{definition}

Assume $\mathbb{D}$ admits small coproducts.

\begin{definition}
\label{Coseparated-precosheaf}A precosheaf $\mathcal{A}$ on a site $X=\left(
Cat\left( X\right) ,Cov\left( X\right) \right) $ with values in $\mathbb{D}$
is called \textbf{coseparated} (\textbf{epiprecosheaf} in the terminology of 
\cite{Bredon-MR0226631}, Section 1, and \cite{Bredon-Book-MR1481706},
Definition VI.1.1), iff%
\begin{equation*}
\dcoprod\limits_{i}\mathcal{A}\left( U_{i}\right) \longrightarrow \mathcal{A}%
\left( U\right)
\end{equation*}%
is an epimorphism for any covering $\left\{ U_{i}\rightarrow U\right\} \in
Cov\left( X\right) $.
\end{definition}

Assume $\mathbb{D}$ is \textbf{cocomplete} (Definition \ref{complete}).

\begin{definition}
\label{Cosheaf in D}A \textbf{cosheaf} on a site $X=\left( Cat\left(
X\right) ,Cov\left( X\right) \right) $ with values in $\mathbb{D}$ is a
precosheaf $\mathcal{A}$ such that the natural morphism%
\begin{equation*}
coker\left( \dcoprod\limits_{i,j}\mathcal{A}\left( U_{i}\times
_{U}U_{j}\right) \rightrightarrows \dcoprod\limits_{i}\mathcal{A}\left(
U_{i}\right) \right) \longrightarrow \mathcal{A}\left( U\right)
\end{equation*}%
is an isomorphism for any covering $\left\{ U_{i}\rightarrow U\right\} \in
Cov\left( X\right) $.

\begin{remark}
For cosheaves of abelian groups, the definition above is given in \cite%
{Bredon-MR0226631}, Section 1, and \cite{Bredon-Book-MR1481706}, Definition
V.1.1. To achieve more generality, and to obtain a more consistent set of
notations, we have replaced the direct sum symbol $\oplus $ in the cited
definitions by the equivalent coproduct symbol $\sqcup $.
\end{remark}
\end{definition}

Let%
\begin{equation*}
X=\left( Cat\left( X\right) ,Cov\left( X\right) \right)
\end{equation*}%
be a site. We introduce the main categories of precosheaves and cosheaves.

\begin{definition}
\label{Categories-of-(pre)cosheaves}Let us denote:

\textbf{a)} by $\mathbb{PCS}\left( X,\mathbb{D}\right) $ the category of
precosheaves on $X$ with values in $\mathbb{D}$;

\textbf{b)} ($\mathbb{D}$ is cocomplete) by $\mathbb{CS}\left( X,\mathbb{D}%
\right) $ the full subcategory of $\mathbb{PCS}\left( X,\mathbb{D}\right) $
consisting of cosheaves.
\end{definition}

\begin{definition}
\label{Ioneda-Cosheaves}Let $\mathbb{D}$ be either $\mathbb{SET}$ or $%
\mathbb{AB}$. Given a precosheaf $\mathcal{A}$ on $X$ with values in $%
Pro\left( \mathbb{D}\right) $, let $\kappa \left( \mathcal{A}\right) $ be
the following presheaf on $X$ with values in $\mathbb{D}^{\mathbb{D}}$:%
\begin{equation*}
\kappa \left( \mathcal{A}\right) \left( U\right) :=\kappa \left( \mathcal{A}%
\left( U\right) \right)
\end{equation*}%
where $\kappa :Pro\left( \mathbb{D}\right) \rightarrow \mathbb{D}^{\mathbb{D}%
}$ is the contravariant embedding from Definition \ref{k-SET}.
\end{definition}

The two Propositions below establish connections between coseparated
precosheaves and separated presheaves, and connections between cosheaves and
sheaves.

\begin{proposition}
\label{Conditions-Coseparated}Let $\mathbb{D}$ be either $\mathbb{SET}$ or $%
\mathbb{AB}$, and let $\mathcal{A}$ be a precosheaf with values in $%
Pro\left( \mathbb{D}\right) $. The following conditions are equivalent:

\begin{enumerate}
\item $\mathcal{A}$ is coseparated;

\item $\kappa \left( \mathcal{A}\right) $ is a separated presheaf with
values in $\mathbb{D}^{\mathbb{D}}$;

\item $\kappa \left( \mathcal{A}\right) \left( Z\right) $ is a separated
presheaf of sets (abelian groups) for any $Z\in \mathbb{D}$.
\end{enumerate}
\end{proposition}

\begin{proof}
$\left( 1\right) \Longleftrightarrow \left( 2\right) $ follows from
Proposition \ref{k-SET-Remark} (\ref{DD-vs-Pro(D)-colimits}, \ref%
{DD-vs-Pro(D)-epi}).

$\left( 2\right) \Longleftrightarrow \left( 3\right) $ follows from
Proposition \ref{k-SET-Remark} (\ref{DD-vs-D-limits}, \ref{DD-vs-D-mono-epi}%
).
\end{proof}

\begin{proposition}
\label{Conditions-Cosheaf}Let $\mathbb{D}$ be either $\mathbb{SET}$ or $%
\mathbb{AB}$, and let $\mathcal{A}$ be a precosheaf with values in $%
Pro\left( \mathbb{D}\right) $. The following conditions are equivalent:

\begin{enumerate}
\item $\mathcal{A}$ is a cosheaf;

\item $\kappa \left( \mathcal{A}\right) $ is a sheaf with values in $\mathbb{%
D}^{\mathbb{D}}$;

\item $\kappa \left( \mathcal{A}\right) \left( Z\right) $ is a sheaf of sets
(abelian groups) for any $Z\in \mathbb{D}$.
\end{enumerate}
\end{proposition}

\begin{proof}
$\left( 1\right) \Longleftrightarrow \left( 2\right) $ follows from
Proposition \ref{k-SET-Remark} (\ref{DD-vs-Pro(D)-colimits}).

$\left( 2\right) \Longleftrightarrow \left( 3\right) $ follows from
Proposition \ref{k-SET-Remark} (\ref{DD-vs-D-limits}).
\end{proof}

\subsection{Cosheaves and precosheaves on topological spaces}

Throughout this Subsection, $X$ is a topological space considered as a site $%
OPEN\left( X\right) $ (see Example \ref{Site-TOP} and Remark \ref%
{Denote-standard-site-simply}).

\begin{proposition}
\label{Prop-Empty-value}

\begin{enumerate}
\item Let $\mathcal{A}$ be a cosheaf with values in $\mathbb{D}$. Then $%
\mathcal{A}\left( \varnothing \right) $ is an initial object in $\mathbb{D}$.

\item Let $\mathcal{A}$ be a coseparated precosheaf with values in $\mathbb{D%
}$ where $\mathbb{D}$ is either $\mathbb{SET}$, or $Pro\left( \mathbb{SET}%
\right) $, or an additive category. Then $\mathcal{A}\left( \varnothing
\right) $ is an initial object in $\mathbb{D}$.
\end{enumerate}
\end{proposition}

\begin{proof}
Let $\left\{ U_{i}\rightarrow \varnothing \right\} _{i\in I}$ be the empty
covering, i.e. the set of indices $I$ is empty. Then, due to Remark \ref%
{Empty-diagram},%
\begin{equation*}
\dcoprod\limits_{i\in I}\mathcal{A}\left( U_{i}\right)
\end{equation*}%
is an initial object in $\mathbb{D}$.

\begin{enumerate}
\item If $\mathcal{A}$ is a cosheaf, then%
\begin{equation*}
\mathcal{A}\left( \varnothing \right) =coker\left( \dcoprod\limits_{i\in
\varnothing }\mathcal{A}\left( U_{i}\right) \rightrightarrows
\dcoprod\limits_{\left( i,j\right) \in \varnothing }\mathcal{A}\left(
U_{i}\cap U_{j}\right) \right) =coker\left( I\rightrightarrows I\right) 
\simeq%
I
\end{equation*}%
where $I$ is the initial object.

\item In the case $\mathcal{A}$ is coseparated, one has an epimorphism%
\begin{equation*}
I=\dcoprod\limits_{i\in \varnothing }\mathcal{A}\left( U_{i}\right)
\longrightarrow \mathcal{A}\left( \varnothing \right) .
\end{equation*}%
If $\mathbb{D}$ is $\mathbb{SET}$ or $Pro\left( \mathbb{SET}\right) $, the
initial object is the empty set, therefore $\mathcal{A}\left( \varnothing
\right) $ is empty as well. If $\mathbb{D}$ is additive, any initial object
is a zero object, therefore $\mathcal{A}\left( \varnothing \right) $ is zero
as well.
\end{enumerate}
\end{proof}

\begin{remark}
There exist categories with an epimorphism $J\rightarrow X$ where $J$ is an
initial object, while $X$ is not.
\end{remark}

\begin{corollary}
\label{Corr-Empty-value}If, in the conditions of Proposition \ref%
{Prop-Empty-value}, $\mathbb{D}$ is $\mathbb{SET}$ or $Pro\left( \mathbb{SET}%
\right) $, then $\mathcal{A}\left( \varnothing \right) =\varnothing $. If $%
\mathbb{D}$ is $\mathbb{AB}$ or $Pro\left( \mathbb{AB}\right) $, then $%
\mathcal{A}\left( \varnothing \right) =0$.
\end{corollary}

\begin{corollary}
\label{AB-cosheaf-is-not-SET-cosheaf}A cosheaf with values in $\mathbb{AB}$
or $Pro\left( \mathbb{AB}\right) $ \textbf{is never a cosheaf} when
considered as a precosheaf with values in $\mathbb{SET}$ or $Pro\left( 
\mathbb{SET}\right) $.
\end{corollary}

\begin{definition}
\label{constant}Let $S$ be a set. We denote by the same letter $S$ the
following \textbf{constant} precosheaf on $X$ with values in $\mathbb{SET}$
or $Pro\left( \mathbb{SET}\right) $: $S\left( U\right) 
{:=}%
S$ for all $U$.
\end{definition}

\begin{definition}
\label{constant-AB}Let $A$ be an abelian group. Analogously to Definition %
\ref{constant}, denote by the same letter $A$ the following precosheaf on $X$
with values in $\mathbb{AB}$ or $Pro\left( \mathbb{AB}\right) $: $A\left(
U\right) 
{:=}%
A$ for all $U$.
\end{definition}

To introduce local isomorphisms, one needs the notion of a \emph{costalk},
which is dual to the notion of a stalk (Definition \ref{Stalk}) in sheaf
theory.

\begin{definition}
\label{Costalk}Let $\mathbb{D}$ be a category, and assume that $\mathbb{D}$
admits cofiltrant limits. Let $x\in X$ be a point. Let further $\mathcal{A}$
be a precosheaf on $X$ with values in $\mathbb{D}$. Denote%
\begin{equation*}
\mathcal{A}^{x}%
{:=}%
\lim_{x\in U}\mathcal{A}\left( U\right) .
\end{equation*}%
We will call $\mathcal{A}^{x}$ the \textbf{costalk} of $\mathcal{A}$ at $x$.
\end{definition}

\begin{example}
If $\mathcal{A}$ is a precosheaf of sets (abelian groups) on $X$, then $%
\mathcal{A}^{x}$ is the limit $\lim_{x\in U}\mathcal{A}\left( U\right) $ in $%
\mathbb{SET}$ ($\mathbb{AB}$). However, if the same precosheaf is considered
as a precosheaf of pro-sets (abelian pro-groups), then $\mathcal{A}^{x}$ is
the pro-set (pro-group) defined by the cofiltrant system%
\begin{equation*}
\mathcal{A}^{x}=\left( \mathcal{A}\left( U\right) :x\in U\right) .
\end{equation*}
\end{example}

\begin{remark}
\label{Costalk-always-in-Pro(D)}Since $\lim_{x\in U}\mathcal{A}\left(
U\right) $ in $\mathbb{SET}$ ($\mathbb{AB}$) is a rather useless object, we
will use notation $\mathcal{A}^{x}$ for $\lim_{x\in U}\mathcal{A}\left(
U\right) $ \textbf{only} in $Pro\left( \mathbb{SET}\right) $ ($Pro\left( 
\mathbb{AB}\right) $), even when the precosheaf $\mathcal{A}$ takes values
in $\mathbb{SET}$ ($\mathbb{AB}$).
\end{remark}

\begin{example}
Let $\mathcal{A}$ is a precosheaf of abelian groups on $X$. According to 
\cite{Bredon-MR0226631}, Section 2, or \cite{Bredon-Book-MR1481706},
Definition V.12.1, $\mathcal{A}$ is called \emph{locally zero} iff for any $%
x\in X$ and any open neighborhood $U$ of $x$ there exists another open
neighborhood $V$, $x\in V\subseteq U$, such that $\mathcal{A}\left( V\right)
\rightarrow \mathcal{A}\left( U\right) $ is zero. If we consider, however,
the precosheaf $\mathcal{A}$ as a precosheaf of abelian pro-groups, then $%
\mathcal{A}$ is locally zero iff for any $x\in X$, $\mathcal{A}^{x}$ is the
zero object in the category $Pro\left( \mathbb{AB}\right) $.
\end{example}

\begin{definition}
A precosheaf $\mathcal{A}$ of abelian (pro-)groups on $X$ is called \textbf{%
locally zero} if $\mathcal{A}^{x}=0$ for any $x\in X$.
\end{definition}

\begin{definition}
\label{local-isomorphism}Let $\mathcal{A\rightarrow B}$ be a morphism of
precosheaves (with values anywhere) on $X$. It is called a \textbf{local
isomorphism} iff $\mathcal{A}^{x}\longrightarrow \mathcal{B}^{x}$ is an
isomorphism for any $x\in X$.
\end{definition}

\begin{proposition}
\label{Local-isomorphism-Bredon}Let $f:\mathcal{A\rightarrow B}$ be a
morphism of precosheaves on $X$ with values in $\mathbb{AB}$ or $Pro\left( 
\mathbb{AB}\right) $. Then $f$ is a local isomorphism iff both $\ker \left(
f\right) $ and $coker\left( f\right) $ are locally zero.
\end{proposition}

\begin{proof}
Since cofiltrant limits are exact (Proposition \ref{k-SET-Remark} (\ref%
{Cofiltrant-limits-exact})), the sequence%
\begin{equation*}
\left( \ker \left( f\right) \right) ^{x}\longrightarrow \mathcal{A}^{x}%
\overset{f^{x}}{\longrightarrow }\mathcal{B}^{x}\longrightarrow \left(
coker\left( f\right) \right) ^{x}
\end{equation*}%
is exact. Since $Pro\left( \mathbb{AB}\right) $ is an abelian category (\cite%
{Kashiwara-MR2182076}, Chapter 8.6), $f^{x}$ is an isomorphism iff both $%
\left( \ker \left( f\right) \right) ^{x}$ and $\left( coker\left( f\right)
\right) ^{x}$ are zero.
\end{proof}

\begin{remark}
It follows from Proposition \ref{Local-isomorphism-Bredon} that a morphism $%
f:\mathcal{A}\rightarrow \mathcal{B}$ of precosheaves of abelian groups is a
local isomorphism in the sense of \cite{Bredon-MR0226631}, Section 3, or 
\cite{Bredon-Book-MR1481706}, Definition V.12.2, iff it is a local
isomorphism in our sense.
\end{remark}

\begin{lemma}
\label{Local-iso-cosheaf-vs-sheaf}Let $\mathcal{A\rightarrow B}$ be a
morphism of precosheaves on $X$ with values in $Pro\left( \mathbb{SET}%
\right) $ ($Pro\left( \mathbb{AB}\right) $). Then it is a local isomorphism
iff%
\begin{equation*}
\left( \kappa \left( \mathcal{B}\right) \left( Z\right) \right)
_{x}\longrightarrow \left( \kappa \left( \mathcal{A}\right) \left( Z\right)
\right) _{x}
\end{equation*}%
is an isomorphism for any set (abelian group) $Z$ and any $x\in X$.
\end{lemma}

\begin{proof}
It follows from the fact that $\kappa $ is a full contravariant embedding
which converts cofiltrant limits to filtrant colimits (Proposition \ref%
{k-SET-Remark} (\ref{DD-vs-Pro(D)-limits})).
\end{proof}

\begin{proposition}
\label{local-isomorphism-cosheaves}Let $\mathbb{D}$ be either $\mathbb{SET}$
or $\mathbb{AB}$, and let $\varphi :\mathcal{A\rightarrow B}$ be a local
isomorphism of cosheaves on $X$ with values in $Pro\left( \mathbb{D}\right) $%
. Then $\varphi $ is an isomorphism.
\end{proposition}

\begin{proof}
Apply Proposition \ref{Sheaves-OPEN(X)-DD} to the morphism of presheaves%
\begin{equation*}
\kappa \left( \varphi \right) :\kappa \left( \mathcal{A}\right)
\longrightarrow \kappa \left( \mathcal{B}\right)
\end{equation*}%
with values in $\mathbb{D}^{\mathbb{D}}$.
\end{proof}

\section{\label{Cosheaves-with-values-in-SET-and-AB}Examples: cosheaves and
precosheaves with values in $\mathbb{SET}$ and $\mathbb{AB}$}

Below is a series of examples of various precosheaves with values anywhere.

\begin{example}
\label{Non-smooth-precosheaf}Let $X=I=\left[ 0,1\right] $, and let $\mathcal{%
A}$ assigns to $U$ the group $S_{1}\left( U,\mathbb{Z}\right) $ of singular $%
1$-chains on $U$. It is proved in \cite{Bredon-MR0226631}, Remark 5.9, and 
\cite{Bredon-Book-MR1481706}, Example VI.5.9, that this precosheaf of
abelian groups is not smooth.
\end{example}

The above example can be improved:

\begin{example}
\label{Singular-precosheaf}Let $\Sigma _{n}\left( \_,A\right) $ be a
precosheaf that assigns to $U$ the colimit of the following sequence:%
\begin{equation*}
S_{n}\left( U,A\right) \overset{\mathbf{ba}}{\longrightarrow }S_{n}\left(
U,A\right) \overset{\mathbf{ba}}{\longrightarrow }S_{n}\left( U,A\right) 
\overset{\mathbf{ba}}{\longrightarrow }...
\end{equation*}%
where $S_{n}\left( U,A\right) $ is the group of singular $A$-valued $n$%
-chains on $U$, and $\mathbf{ba}$ is the barycentric subdivision. It is
proved in \cite{Bredon-MR0226631}, Section 10, and \cite%
{Bredon-Book-MR1481706}, Proposition VI.12.1, that $\Sigma _{n}\left(
\_,A\right) $ is a cosheaf of abelian groups (and of abelian pro-groups, due
to Theorem \ref{Our-cosheaves-vs-Bredon}).
\end{example}

\begin{example}
\label{p0-prime-not-cosheaf}Let $\pi $ be a precosheaf of sets that assigns
to $U$ the set $\pi \left( U\right) $ of connected components of $U$. This
precosheaf is coseparated. If $X$ is \textbf{locally connected}, then, for
any open subset $U\subseteq X$, the pro-homotopy set $pro$-$\pi _{0}\left(
U\right) $ is isomorphic to the \textbf{trivial} (Remark \ref%
{Trivial-pro-object}) pro-set $\pi \left( U\right) $. It follows from
Theorem \ref{Our-cosheaves-vs-Bredon}, that $\pi 
\simeq%
\left( \mathbf{pt}\right) _{\#}$ where $\mathbf{pt}$ is the one-point
constant precosheaf. Therefore, $\mathbf{pt}$ is smooth, and $\pi $ a
constant cosheaf (compare to \cite{Bredon-MR0226631}, Remark 5.11).

In general, if $X$ is \textbf{not} locally connected, $\pi $ is \textbf{not
a cosheaf}. Indeed, let 
\begin{equation*}
X=Y\cup Z\subseteq \mathbb{R}^{2},
\end{equation*}%
where $Y$ is the line segment between the points $\left( 0,1\right) $ and $%
\left( 0,-1\right) $, and $Y$ is the graph of $y=\sin \left( \frac{1}{x}%
\right) $ for $0<x\leq 2\pi $. Let further%
\begin{eqnarray*}
X &=&U=U_{1}\cup U_{2}, \\
U_{1} &=&\left\{ \left( x,y\right) \in X~|~y>-\frac{1}{2}\right\} , \\
U_{2} &=&\left\{ \left( x,y\right) \in X~|~y<\frac{1}{2}\right\} .
\end{eqnarray*}%
$X$ is a connected (not locally connected!) compact metric space. Take $%
P=\left( 0,1\right) \in U_{1}$ and $Q=\left( \frac{3\pi }{2},-1\right) \in
U_{2}$. Since $U=X$ is connected, these two points are mapped to the same
point of $U$ under the canonical mapping%
\begin{equation*}
U_{1}\sqcup U_{2}\longrightarrow U.
\end{equation*}%
However, these two points define \textbf{different} elements of the colimit%
\begin{equation*}
coker\left( \pi \left( U_{1}\cap U_{2}\right) \rightrightarrows \pi \left(
U_{1}\right) \sqcup \pi \left( U_{1}\right) \right) .
\end{equation*}%
Therefore,%
\begin{equation*}
coker\left( \pi \left( U_{1}\cap U_{2}\right) \rightrightarrows \pi \left(
U_{1}\right) \sqcup \pi \left( U_{1}\right) \right) \longrightarrow \pi
\left( U\right) =\pi \left( X\right)
\end{equation*}%
is not injective, and $\pi $ is not a cosheaf.

See also Example \ref{Converging-sequence}.
\end{example}

\begin{example}
\label{p0-is-cosheaf}Let $\pi _{0}$ be a precosheaf of sets that assigns to $%
U$ the set $\pi _{0}\left( U\right) $ of path-connected components of $U$.
Then $\pi _{0}$ is a cosheaf of sets (and of pro-sets, due to Theorem \ref%
{Our-cosheaves-vs-Bredon}). This cosheaf is constant if $X$ is locally
path-connected, and is not constant in general. Indeed, $\pi _{0}$ is
clearly coseparated. Let $\left\{ U_{i}\rightarrow U\right\} _{i\in I}$ be
an open covering, and let $P\in U_{s}$ and $Q\in U_{t}$ be two points lying
in the same path-connected component. Therefore, there exists a continuous
path%
\begin{equation*}
g:\left[ 0,1\right] \longrightarrow U
\end{equation*}%
with $g\left( 0\right) =P$ and $g\left( 1\right) =Q$. Using Lebesgue's
Number Lemma, one proves that $P$ and $Q$ define equal elements of the
cokernel below. Therefore, the mapping%
\begin{equation*}
coker\left( \dcoprod\limits_{i,j}\pi _{0}\left( U_{i}\cap U_{j}\right)
\rightrightarrows \dcoprod\limits_{i}\pi _{0}\left( U_{i}\right) \right)
\longrightarrow \pi _{0}\left( U\right)
\end{equation*}%
is injective, thus bijective, and $\pi _{0}$ is a cosheaf.
\end{example}

\begin{example}
\label{H0-is-cosheaf-in-AB}Let $A$ be an abelian group, and let $%
H_{0}^{S}\left( \_,A\right) $ be the precosheaf of abelian groups that
assigns to $U$ the zeroth singular homology group $H_{0}^{S}\left(
X,A\right) $. Then $H_{0}^{S}\left( \_,A\right) $ is a cosheaf. Indeed,%
\begin{equation*}
H_{0}^{S}=coker\left( \Sigma _{1}\left( \_,A\right) \longrightarrow \Sigma
_{0}\left( \_,A\right) \right)
\end{equation*}%
where $\Sigma _{n}\left( \_,A\right) $ is the cosheaf from Example \ref%
{Singular-precosheaf}. The embedding%
\begin{equation*}
\mathbb{CS}\left( X,Pro\left( \mathbb{AB}\right) \right) \longrightarrow 
\mathbb{PCS}\left( X,Pro\left( \mathbb{AB}\right) \right) ,
\end{equation*}%
being the left adjoint to $\left( {}\right) _{\#}$, commutes with colimits.
Therefore, $H_{0}^{S}\left( \_,A\right) $ is a cosheaf because $\Sigma
_{1}\left( \_,A\right) $ and $\Sigma _{0}\left( \_,A\right) $ are cosheaves. 
$H_{0}^{S}\left( \_,A\right) $ is constant if $X$ is locally path-connected.
However, $H_{0}^{S}\left( \_,A\right) $ is not constant in general, see
Example \ref{Converging-sequence}.
\end{example}

\begin{example}
\label{Converging-sequence}Let $X$ be the following sequence converging to
zero (together with the limit):%
\begin{equation*}
X=\left\{ 0\right\} \cup \left\{ 1,\frac{1}{2},\frac{1}{3},\frac{1}{4}%
,...\right\} \subseteq \mathbb{R}.
\end{equation*}%
The precosheaves $\pi $ and $\pi _{0}$ from Examples \ref%
{p0-prime-not-cosheaf} and \ref{p0-is-cosheaf} coincide on $X$. Therefore, $%
\pi =\pi _{0}$ is a cosheaf. However, it is not constant. To see this, just
compare the costalks at different points $x\in X$: $\left( \pi \right)
_{x}=\left\{ \mathbf{pt}\right\} $ if $x\neq 0$, while $\left( \pi \right)
_{0}$ is a \textbf{non-trivial} (Remark \ref{Trivial-pro-object}) pro-set.
Consider the constant precosheaf $\mathbf{pt}$. Due to Theorem \ref%
{Main-constant}, $\left( \mathbf{pt}\right) _{\#}%
\simeq%
pro$-$\pi _{0}$. The latter cosheaf does \textbf{not} take values in $%
\mathbb{SET}$,\ therefore, due to Theorem \ref{Our-cosheaves-vs-Bredon}, the
precosheaf $\mathbf{pt}$ is \textbf{not} smooth. Similarly, it can be
proved, that the cosheaf $H_{0}^{S}\left( \_,A\right) $ from Example \ref%
{H0-is-cosheaf-in-AB} is not constant on $X$, while the constant precosheaf $%
A$ is not smooth, because $\left( A\right) _{\#}%
\simeq%
pro$-$H_{0}\left( \_,A\right) $ does not take values in $\mathbb{AB}$.
\end{example}

\begin{example}
\label{H0-is-cosheaf}Let $X$ is locally connected and locally compact
Hausdorff space, and let $\mathcal{A}$ be a sheaf. Then $H_{0}^{c}\left( X,%
\mathcal{A}\right) $ where $H_{0}^{c}$ is the Borel-Moore homology with
compact supports, is a cosheaf. If $X$ is locally path-connected, and $%
\mathcal{A}$ is constant, then this cosheaf is constant. See \cite%
{Bredon-Book-MR1481706}, Proposition VI.12.2.
\end{example}

\section{\label{Proofs}Proofs of the main results}

\subsection{\label{Plus-construction}\label{Cosheafification}Cosheafification%
}

\begin{definition}
Let $\left\{ U_{i}\rightarrow U\right\} \in Cov\left( X\right) $, and let $%
\mathcal{A}$ be a precosheaf with values in $Pro\left( \mathbb{SET}\right) $
($Pro\left( \mathbb{AB}\right) $). Define%
\begin{equation*}
H_{0}\left( \left\{ U_{i}\longrightarrow U\right\} ,\mathcal{A}\right)
:=coker\left( \dcoprod\limits_{i,j}\mathcal{A}\left( U_{i}\times
_{U}U_{j}\right) \rightrightarrows \dcoprod\limits_{i}\mathcal{A}\left(
U_{i}\right) \right) .
\end{equation*}%
The definition is correct since both $Pro\left( \mathbb{SET}\right) $ and $%
Pro\left( \mathbb{AB}\right) $ are cocomplete.
\end{definition}

\begin{proposition}
\label{H0-vs-H0}%
\begin{equation*}
\kappa \left( H_{0}\left( \left\{ U_{i}\longrightarrow U\right\} ,A\right)
\right) =\ker \left( \dprod\limits_{i}\kappa \left( \mathcal{A}\left(
U_{i}\right) \right) \rightrightarrows \dprod\limits_{i,j}\kappa \left( 
\mathcal{A}\left( U_{i}\times _{U}U_{j}\right) \right) \right) =H^{0}\left(
\left\{ U_{i}\longrightarrow U\right\} ,\kappa \left( \mathcal{A}\right)
\right)
\end{equation*}%
where $H^{0}$ is the functor from Definition \ref{H0-sheaves}.
\end{proposition}

\begin{proof}
The functor $\kappa $ converts colimits (e.g., coproducts and cokernels) to
limits (e.g., products and kernels).
\end{proof}

\begin{lemma}
Given two coverings $\mathcal{V},\mathcal{U}\in Cov\left( X\right) $, and
two refinement mappings $f,g:\mathcal{V\rightarrow U}$, the corresponding
mappings of cokernels coincide:%
\begin{equation*}
H_{0}\left( f,\mathcal{A}\right) =H_{0}\left( g,\mathcal{A}\right)
:H_{0}\left( \mathcal{V},\mathcal{A}\right) \longrightarrow H_{0}\left( 
\mathcal{U},\mathcal{A}\right) .
\end{equation*}
\end{lemma}

\begin{proof}
Apply Proposition \ref{H0-vs-H0} and Lemma \ref{Two-refinements}.
\end{proof}

\begin{definition}
Given $U\in Cat\left( X\right) $, the set of coverings on $U$ is a
cofiltrant pre-ordered set under the refinement relation: $\mathcal{V}\leq 
\mathcal{U}$ iff $\mathcal{V}$ refines $\mathcal{U}$. Let us define%
\begin{equation*}
\mathcal{A}_{+}\left( U\right) :=\lim_{\mathcal{V}}H_{0}\left( \mathcal{V},%
\mathcal{A}\right)
\end{equation*}%
where $\mathcal{V}$ runs over coverings on $U$. $\mathcal{A}_{+}$ is clearly
a precosheaf in $Pro\left( \mathbb{SET}\right) $ ($Pro\left( \mathbb{AB}%
\right) $).
\end{definition}

\begin{proposition}
\label{Plus-cosheaf-vs-sheaf}%
\begin{equation*}
\kappa \left( \mathcal{A}_{+}\right) 
\simeq%
\left( \kappa \left( \mathcal{A}\right) \right) ^{+}.
\end{equation*}
\end{proposition}

\begin{proof}
Follows from Propositions \ref{H0-vs-H0} and \ref{k-SET-Remark} (\ref%
{DD-vs-Pro(D)-colimits}, \ref{DD-vs-Pro(D)-limits}).
\end{proof}

\begin{proposition}
\label{Properties-Plus-construction-cosheaves}Let $\mathbb{D}$ be either $%
\mathbb{SET}$ or $\mathbb{AB}$, and let $\mathcal{A}$ be a cosheaf on $X$
with values in $Pro\left( \mathbb{D}\right) $. Then:

\begin{enumerate}
\item $\mathcal{A}_{+}$ is coseparated.

\item If $\mathcal{A}$ is coseparated, then $\mathcal{A}_{+}$ is a cosheaf.

\item The functor%
\begin{equation*}
\left( {}\right) _{\#}:=\left( {}\right) _{++}:\mathbb{PCS}\left(
X,Pro\left( \mathbb{D}\right) \right) \longrightarrow \mathbb{CS}\left(
X,Pro\left( \mathbb{D}\right) \right)
\end{equation*}%
is right adjoint to the inclusion functor.
\end{enumerate}
\end{proposition}

\begin{proof}
Apply Proposition \ref{Properties-Plus-construction} to $\kappa \left( 
\mathcal{A}\right) $.
\end{proof}

\subsection{Proof of Theorem \protect\ref{Main-Pro-Set}}

\begin{proof}
Let $\mathbb{D}$ be either $\mathbb{SET}$ or $\mathbb{AB}$. The right
adjointness of the functor%
\begin{equation*}
\left( {}\right) _{\#}:\mathbb{P}\mathbb{CS}\left( X,Pro\left( \mathbb{D}%
\right) \right) \longrightarrow \mathbb{CS}\left( X,Pro\left( \mathbb{D}%
\right) \right)
\end{equation*}%
is already proved (Proposition \ref{Properties-Plus-construction-cosheaves}%
). It remains only to prove that $\left( {}\right) _{\#}$ is a reflector.
Let $\mathcal{A}$ be a cosheaf with values in $Pro\left( \mathbb{D}\right) $%
. Let further $\mathcal{B}$ be an arbitrary cosheaf with values in $%
Pro\left( \mathbb{D}\right) $. It is enough to prove that%
\begin{equation*}
Hom_{\mathbb{P}\mathbb{CS}\left( X,\mathbb{D}\right) }\left( \mathcal{B},%
\mathcal{A}_{\#}\right) 
\simeq%
Hom_{\mathbb{CS}\left( X,\mathbb{D}\right) }\left( \mathcal{B},\mathcal{A}%
\right)
\end{equation*}%
naturally on $\mathcal{B}$. There exist natural on $\mathcal{B}$ isomorphisms%
\begin{eqnarray*}
&&Hom_{\mathbb{P}\mathbb{CS}\left( X,\mathbb{D}\right) }\left( \mathcal{B},%
\mathcal{A}_{\#}\right) 
\simeq%
Hom_{\mathbb{CS}\left( X,\mathbb{D}\right) }\left( \mathcal{B},\mathcal{A}%
_{\#}\right) 
\simeq
\\
&&Hom_{\mathbb{P}\mathbb{CS}\left( X,\mathbb{D}\right) }\left( \mathcal{B},%
\mathcal{A}\right) 
\simeq%
Hom_{\mathbb{CS}\left( X,\mathbb{D}\right) }\left( \mathcal{B},\mathcal{A}%
\right) .
\end{eqnarray*}%
The first and the last isomorphisms are due to the full embedding of $%
\mathbb{CS}\left( X,\mathbb{D}\right) $ into $\mathbb{P}\mathbb{CS}\left( X,%
\mathbb{D}\right) $, while the second isomorphism is the adjunction. It
follows that $\mathcal{A}_{\#}%
\simeq%
\mathcal{A}$.
\end{proof}

\begin{remark}
The reasoning above can be generalized to any full embedding 
\begin{equation*}
I:\mathbb{E}\longrightarrow \mathbb{F}.
\end{equation*}%
If such an embedding has a right or a left adjoint $G$, then $G$ is clearly
a reflector.
\end{remark}

\subsection{Proof of Theorem \protect\ref{Main-local-iso-pro-SET}}

\begin{proof}
(2) is already proved (Proposition \ref{local-isomorphism-cosheaves}).

To prove (1), apply Proposition \ref{Sheaves-OPEN(X)-DD}(2) to $\kappa
\left( \mathcal{A}\right) $.

To prove (3), consider the composition%
\begin{equation*}
\mathcal{B}\longrightarrow \mathcal{A}_{\#}\longrightarrow \mathcal{A},
\end{equation*}%
existing due to the right adjointness of $\left( {}\right) _{\#}$. The
composition and the morphism $\mathcal{A}_{\#}\longrightarrow \mathcal{A}$
are local isomorphisms, therefore $\mathcal{B}\longrightarrow \mathcal{A}%
_{\#}$ is a local isomorphism between cosheaves, hence an isomorphism.
\end{proof}

\subsection{Proof of Theorem \protect\ref{Our-cosheaves-vs-Bredon}}

\begin{proof}
\begin{enumerate}
\item The full embedding $\mathbb{D}\hookrightarrow Pro\left( \mathbb{D}%
\right) $ commutes with small colimits (e.g., cokernels and small
coproducts), and preserves epimorphisms, i.e. $f:X\rightarrow Y$ is an
epimorphism in $\mathbb{D}$ iff $f$ is an epimorphism in $Pro\left( \mathbb{D%
}\right) $ (see Proposition \ref{k-SET-Remark} (\ref%
{Pro(D)-vs-D-limits-colimits}, \ref{Pro(D)-vs-D-mono-epi})).

\item \textbf{Only if} part. Assume $\mathcal{A}$ is smooth. Then there
exists local isomorphisms%
\begin{equation*}
\mathcal{A}\longrightarrow \mathcal{B}\longleftarrow \mathcal{C}
\end{equation*}%
of precosheaves with values in $\mathbb{D}$, where $\mathcal{C}$ is a
cosheaf. Apply $\left( {}\right) _{\#}$ and obtain the following diagram of
isomorphisms of cosheaves with values in $Pro\left( \mathbb{D}\right) $:%
\begin{equation*}
\mathcal{A}_{\#}\longrightarrow \mathcal{B}_{\#}\longleftarrow \mathcal{C}%
_{\#}%
\simeq%
\mathcal{C}.
\end{equation*}%
It follows that $\mathcal{A}_{\#}%
\simeq%
\mathcal{C}$, and therefore $\mathcal{A}_{\#}$ takes values in $\mathbb{D}$.

\textbf{If} part. Assume $\mathcal{A}$ takes values in $\mathbb{D}$. Then $%
\mathcal{A}_{\#}\longrightarrow \mathcal{A}$ is the desired local
isomorphism.
\end{enumerate}
\end{proof}

\subsection{Pro-homotopy and pro-homology}

Let $\mathbb{TOP}\ $be the category of topological spaces and continuous
mappings. There are following categories closely connected to $\mathbb{TOP}$%
: the category $H\left( \mathbb{TOP}\right) $ of homotopy types, the
category $Pro\left( H\left( \mathbb{TOP}\right) \right) $ of pro-homotopy
types, and the category $H\left( Pro\left( \mathbb{TOP}\right) \right) $ of
homotopy types of pro-spaces. The last category is used in \emph{strong
shape theory}. It is finer than the second one which is used in \emph{shape
theory}. One of the most important tools in strong shape theory is a \emph{%
strong expansion} (see \cite{Mardesic-MR1740831}, conditions (S1) and (S2)
on p. 129). In this paper, it is sufficient to use a weaker notion: an $%
H\left( \mathbb{TOP}\right) $-\emph{expansion} (\cite%
{Mardesic-Segal-MR676973}, \S I.4.1, conditions (E1) and (E2)). Those two
conditions are equivalent to the following

\begin{definition}
\label{HTOP-extension}Let $X$ be a topological space. A morphism in $%
Pro\left( H\left( \mathbb{TOP}\right) \right) $%
\begin{equation*}
X\longrightarrow \left( Y_{j}\right)
\end{equation*}%
is called an $H\left( \mathbb{TOP}\right) $-expansion (or simply \emph{%
expansion}) if for any polyhedron (equivalently, an ANR) $P$ the following
mapping%
\begin{equation*}
\underset{j}{colim}~\left[ Y_{j},P\right] =\underset{j}{colim}~Hom_{H\left( 
\mathbb{TOP}\right) }\left( Y_{j},P\right) \longrightarrow Hom_{H\left( 
\mathbb{TOP}\right) }\left( X,P\right) =\left[ X,P\right]
\end{equation*}%
is bijective where $\left[ Z,P\right] $ is the set of homotopy classes of
continuous mappings from $Z$ to $P$.

An expansion is called \textbf{polyhedral} (or an $H\left( \mathbb{POL}%
\right) $-expansion) if all $Y_{j}$ are polyhedra. It is called an $\mathbb{%
ANR}$-expansion if all $Y_{j}$ are $ANR$s.
\end{definition}

Pro-homotopy is defined in \cite{Mardesic-Segal-MR676973}, p. 121:

\begin{definition}
\label{Pro-homotopy-groups}For a (pointed) topological space $X$, define its
pro-homotopy pro-sets%
\begin{equation*}
pro\text{-}\pi _{n}\left( X\right) 
{:=}%
\left( \pi _{n}\left( Y_{j}\right) \right)
\end{equation*}%
where $X\rightarrow \left( Y_{j}\right) $ is an $H\left( \mathbb{POL}\right) 
$- or an $\mathbb{ANR}$-expansion.
\end{definition}

\begin{remark}
As in the \textquotedblleft usual\textquotedblright\ algebraic topology, $%
pro $-$\pi _{0}$ is a pro-set, $pro$-$\pi _{1}$ is a pro-group, and $pro$-$%
\pi _{n}$ are abelian pro-groups for $n\geq 2$.
\end{remark}

Pro-homology groups are defined in \cite{Mardesic-Segal-MR676973}, \S %
II.3.2, as follows:

\begin{definition}
\label{Pro-homology-groups}For a topological space $X$, and an abelian group 
$A$, define its pro-homology groups as%
\begin{equation*}
pro\text{-}H_{n}\left( X,A\right) 
{:=}%
\left( H_{n}\left( Y_{j},A\right) \right)
\end{equation*}%
where $X\rightarrow \left( Y_{j}\right) $ is an $H\left( \mathbb{POL}\right) 
$- or an $\mathbb{ANR}$-expansion.
\end{definition}

We will use below the notion of \v{C}ech cohomology $\check{H}^{\ast }$
defined in \cite{Mardesic-Segal-MR676973}, \S II.3.2. as follows:

\begin{definition}
\label{Cech-cohomology}Let $X$ be a topological space, and $A$ be an abelian
group. Then%
\begin{equation*}
\check{H}^{n}\left( X,A\right) 
{:=}%
\underset{j}{colim}~H^{n}\left( Y_{j},A\right)
\end{equation*}%
where $X\longrightarrow \left( Y_{j}\right) $ is an $H\left( \mathbb{POL}%
\right) $- or an $\mathbb{ANR}$-expansion.
\end{definition}

\begin{remark}
\label{Shape-H0-vs-Cech-H0}The above definition is sufficient for our
purposes. It should be mentioned, however, how this definition is related to
the \textquotedblleft usual\textquotedblright\ one, like in \cite%
{Bredon-Book-MR1481706}, \S I.7 and \S III.4.

Theorem 8 from \cite{Mardesic-Segal-MR676973}, App.1, \S 3.2, shows that an $%
H\left( \mathbb{POL}\right) $-expansion for $X$ can be constructed using
nerves of normal open coverings of $X$. It follows that our definition is
very similar to the usual one, except for the coverings used. In the usual
one all the open coverings are considered. Therefore, our definition is a
variant of the usual one for the site $NORM\left( X\right) $ (Example \ref%
{Site-NORM}) instead of the standard site $OPEN\left( X\right) $ (Example %
\ref{Site-TOP}). It can be proved that the two definitions coincide if $X$
is a paracompact Hausdorff space. However, on the level of $\check{H}^{0}$
the two definitions coincide for an \textbf{arbitrary} space $X$. This fact
is not proved (is not needed either) in this paper. However, a glimpse of
the proof can be found in Proposition \ref{Dual-pro-H0}(2) below.
\end{remark}

\begin{proposition}
\label{Dual-pro-p0}\label{Dual-pro-H0}~

\begin{enumerate}
\item For any set $Z$ and any topological space $U$, the set%
\begin{equation*}
Hom_{\mathbb{SET}}\left( S\times pro\text{-}\pi _{0}\left( U\right) ,Z\right)
\end{equation*}%
is naturally (with respect to $S$, $Z$ and $U$) isomorphic to the set $%
Z^{S\times U}$ of continuous functions $S\times U\rightarrow Z$ where $S$
and $Z$ are supplied with the discrete topology.

\item For any abelian group $Z$ and any topological space $U$, the set%
\begin{equation*}
Hom_{Pro\left( \mathbb{AB}\right) }\left( pro\text{-}H_{0}\left( U,A\right)
,Z\right)
\end{equation*}%
is naturally (with respect to $A$, $Z$ and $U$) isomorphic to the \v{C}ech
cohomology group%
\begin{equation*}
\check{H}^{0}\left( U,Hom_{\mathbb{AB}}\left( A,Z\right) \right)
\end{equation*}%
which, in turn, is isomorphic to the group $\left( Hom_{\mathbb{AB}}\left(
A,Z\right) \right) ^{U}$ of continuous functions 
\begin{equation*}
U\longrightarrow Hom_{\mathbb{AB}}\left( A,Z\right)
\end{equation*}%
where $Hom_{\mathbb{AB}}\left( A,Z\right) $ is supplied with the discrete
topology.
\end{enumerate}
\end{proposition}

\begin{proof}
The two statements are proved similarly. The proof of (2) is a bit more
complicated, and is given below. The proof of (1) is left to the reader.

Let $U\rightarrow \left( Y_{j}\right) $ be a polyhedral (or $ANR$)
expansion, and let%
\begin{equation*}
pro\text{-}H_{0}\left( Y_{j},A\right) =\left( H_{0}\left( Y_{j},A\right)
\right)
\end{equation*}%
be the corresponding abelian pro-group. Since the spaces $Y_{j}$ are locally
path-connected, 
\begin{equation*}
H_{0}\left( Y_{j},A\right) =\dbigoplus\limits_{\pi _{0}\left( Y_{j}\right) }A
\end{equation*}%
and%
\begin{equation*}
Hom_{\mathbb{TOP}}\left( Y_{j},V\right) =Hom_{\mathbb{SET}}\left( \pi
_{0}\left( Y_{j}\right) ,V\right)
\end{equation*}%
for any discrete topological space $V$. Since $Hom_{\mathbb{AB}}\left(
A,Z\right) $ is considered as a discrete topological space, one has a
sequence of isomorphisms%
\begin{eqnarray*}
&&Hom_{Pro\left( \mathbb{AB}\right) }\left( \left( H_{0}\left(
Y_{j},A\right) \right) ,Z\right) 
\simeq%
\underset{j}{colim}~Hom_{\mathbb{AB}}\left( \left( \dbigoplus\limits_{\pi
_{0}\left( Y_{j}\right) }A\right) ,Z\right) 
\simeq
\\
&&%
\simeq%
\underset{j}{colim}~\dprod\limits_{\pi _{0}\left( Y_{j}\right) }Hom_{\mathbb{%
AB}}\left( A,Z\right) 
\simeq%
\underset{j}{colim}~Hom_{\mathbb{SET}}\left( \pi _{0}\left( Y_{j}\right)
,Hom_{\mathbb{AB}}\left( A,Z\right) \right) 
\simeq
\\
&&%
\simeq%
\underset{j}{colim}~Hom_{\mathbb{TOP}}\left( Y_{j},Hom_{\mathbb{AB}}\left(
A,Z\right) \right) .
\end{eqnarray*}%
The compositions%
\begin{equation*}
U\longrightarrow Y_{j}\longrightarrow Hom_{\mathbb{AB}}\left( A,Z\right)
\end{equation*}%
define a natural mapping%
\begin{equation*}
\underset{j}{colim}~\left[ Y_{j},Hom_{\mathbb{AB}}\left( A,Z\right) \right]
\longrightarrow \left[ U,Hom_{\mathbb{AB}}\left( A,Z\right) \right] .
\end{equation*}%
That mapping is an isomorphism because $U\rightarrow \left( Y_{j}\right) $
is an expansion. Since $Hom_{\mathbb{AB}}\left( A,Z\right) $ is discrete,
the homotopy classes of mappings%
\begin{equation*}
U\longrightarrow Hom_{\mathbb{AB}}\left( A,Z\right)
\end{equation*}%
and%
\begin{equation*}
Y_{j}\longrightarrow Hom_{\mathbb{AB}}\left( A,Z\right)
\end{equation*}%
consist of single mappings, therefore both%
\begin{eqnarray*}
\underset{j}{colim}~\left[ Y_{j},Hom_{\mathbb{AB}}\left( A,Z\right) \right]
&=&\underset{j}{colim}~Hom_{\mathbb{TOP}}\left( Y_{j},Hom_{\mathbb{AB}%
}\left( A,Z\right) \right) \longrightarrow \\
&\longrightarrow &Hom_{\mathbb{TOP}}\left( U,Hom_{\mathbb{AB}}\left(
A,Z\right) \right) =\left[ U,Hom_{\mathbb{AB}}\left( A,Z\right) \right] ,
\end{eqnarray*}%
and the internal mapping%
\begin{equation*}
Hom_{Pro\left( \mathbb{AB}\right) }\left( \left( H_{0}\left( Y_{j},A\right)
\right) ,Z\right) 
\simeq%
\underset{j}{colim}~Hom_{\mathbb{TOP}}\left( Y_{j},Hom_{\mathbb{AB}}\left(
A,Z\right) \right) \longrightarrow Hom_{\mathbb{TOP}}\left( U,Hom_{\mathbb{AB%
}}\left( A,Z\right) \right)
\end{equation*}%
are isomorphisms. The \v{C}ech cohomology group%
\begin{equation*}
\check{H}^{0}\left( U,Hom_{\mathbb{AB}}\left( A,Z\right) \right)
\end{equation*}%
is isomorphic to 
\begin{eqnarray*}
&&Hom_{Pro\left( \mathbb{AB}\right) }\left( \left( H_{0}\left( Y_{j}\right)
\right) ,Hom_{\mathbb{AB}}\left( A,Z\right) \right) 
\simeq%
\underset{j}{colim}~Hom_{\mathbb{AB}}\left( H_{0}\left( Y_{j}\right) ,Hom_{%
\mathbb{AB}}\left( A,Z\right) \right) 
\simeq
\\
&&%
\simeq%
\underset{j}{colim}~Hom_{\mathbb{AB}}\left( H_{0}\left( Y_{j}\right) \otimes
A,Z\right) 
\simeq%
\underset{j}{colim}~Hom_{\mathbb{AB}}\left( \left( \dbigoplus\limits_{\pi
_{0}\left( Y_{j}\right) }A\right) ,Z\right) 
\simeq
\\
&&%
\simeq%
Hom_{Pro\left( \mathbb{AB}\right) }\left( \left( H_{0}\left( Y_{j},A\right)
\right) ,Z\right) .
\end{eqnarray*}
\end{proof}

\subsection{Proof of Theorem \protect\ref{Main-constant}}

\begin{proof}
\textbf{(1)} Let $Z$ be a set. Then, due to Proposition \ref{Dual-pro-p0}%
(1), the presheaf%
\begin{equation*}
\kappa \left( \mathcal{P}\right) \left( Z\right) =\kappa \left( S\times pro%
\text{-}\pi _{0}\right) \left( Z\right)
\end{equation*}%
of sets is isomorphic to the presheaf $\mathcal{B}$, $\mathcal{B}\left(
U\right) :=Z^{S\times U}$. For any open covering $\left\{ U_{i}\rightarrow
U\right\} $ the space $S\times U$ ($S$ with the discrete topology) is
isomorphic in the category $\mathbb{TOP}$ to the cokernel%
\begin{equation*}
coker\left( \dcoprod\limits_{i,j}\left( S\times \left( U_{i}\cap
U_{j}\right) \right) \rightrightarrows \dcoprod\limits_{i,j}\left( S\times
U_{i}\right) \right) ,
\end{equation*}%
therefore%
\begin{eqnarray*}
Z^{S\times U} &=&Hom_{\mathbb{TOP}}\left( S\times U,Z\right) 
\simeq
\\
&&%
\simeq%
\ker \left( Hom_{\mathbb{TOP}}\left( \dcoprod\limits_{i,j}\left( S\times
U_{i}\right) ,Z\right) \rightrightarrows Hom_{\mathbb{TOP}}\left(
\dcoprod\limits_{i,j}\left( S\times \left( U_{i}\cap U_{j}\right) \right)
,Z\right) \right) 
\simeq
\\
&&%
\simeq%
\ker \left( \dprod\limits_{i}\mathcal{B}\left( U_{i}\right)
\rightrightarrows \dprod\limits_{i,j}\mathcal{B}\left( U_{i}\cap
U_{j}\right) \right) ,
\end{eqnarray*}%
and $\mathcal{B}$ is a sheaf of sets. Proposition \ref{Conditions-Cosheaf}
implies that $\mathcal{P}=S\times pro$-$\pi _{0}$ is a cosheaf.

\textbf{(2)} It is enough to prove that%
\begin{equation*}
\mathcal{P}=S\times pro\text{-}\pi _{0}\longrightarrow S
\end{equation*}%
is a local isomorphism. Let $Z$ be a set, and let $x\in X$. Clearly $\kappa
\left( S\right) \left( Z\right) _{x}%
\simeq%
Z^{S}$. Moreover,%
\begin{equation*}
\left( \kappa \left( \mathcal{P}\right) \right) _{x}=\left( \kappa \left(
S\times pro\text{-}\pi _{0}\right) \left( Z\right) \right) _{x}=\underset{%
x\in V}{colim}~Z^{S\times V}
\end{equation*}%
where the colimit is taken over all open neighborhoods $V$ of $x$. The
mappings $S\times V\rightarrow Z$ involved are locally constant since $Z$ is
discrete. Therefore, any two germs $\left[ f\right] $ and $\left[ g\right] $%
, $f,g:W\rightarrow Z$, where $W$ is an open neighborhood of $S\times \{x\}$%
, are equivalent iff $f|_{S\times \{x\}}=g|_{S\times \{x\}}$. It follows
that 
\begin{equation*}
\underset{x\in V}{colim}~Z^{S\times V}%
\simeq%
Z^{S},
\end{equation*}%
and that both the mapping of presheaves%
\begin{equation*}
\kappa \left( S\right) \longrightarrow \kappa \left( S\times pro\text{-}\pi
_{0}\right) \left( Z\right)
\end{equation*}%
and the mapping of precosheaves%
\begin{equation*}
\mathcal{P}=S\times pro\text{-}\pi _{0}\longrightarrow S
\end{equation*}%
are local isomorphisms. Due to Theorem \ref{Main-local-iso-pro-SET}, 
\begin{equation*}
\mathcal{P}=S\times pro\text{-}\pi _{0}%
\simeq%
\left( S\right) _{\#}.
\end{equation*}

\textbf{(3)} Let $Z$ be an abelian group. Proposition \ref{Dual-pro-H0}(2)
implies that the presheaf $\kappa \left( pro\text{-}H_{0}\left( \_,A\right)
\right) \left( Z\right) $ of abelian groups is isomorphic to the presheaf $%
\mathcal{C}$, $\mathcal{C}\left( U\right) :=\left( Hom_{\mathbb{AB}}\left(
A,Z\right) \right) ^{U}$. For any open covering $\left\{ U_{i}\rightarrow
U\right\} $ the space $U$ is isomorphic in the category $\mathbb{TOP}$ to
the cokernel%
\begin{equation*}
coker\left( \dcoprod\limits_{i,j}\left( U_{i}\cap U_{j}\right)
\rightrightarrows \dcoprod\limits_{i,j}U_{i}\right) ,
\end{equation*}%
therefore, reasoning as in (1), on gets that $\mathcal{C}$ is a sheaf of
abelian groups. It follows that $\mathcal{H}=pro$-$H_{0}\left( \_,A\right) $
is a cosheaf.

\textbf{(4)} It is enough to prove that%
\begin{equation*}
\mathcal{H}=pro\text{-}H_{0}\left( \_,A\right) \longrightarrow A
\end{equation*}%
is a local isomorphism. Let $Z$ be an abelian group, and let $x\in X$.
Reasoning similarly to (2), one gets isomorphisms%
\begin{eqnarray*}
&&\kappa \left( A\right) \left( Z\right) _{x}%
\simeq%
Hom_{\mathbb{AB}}\left( A,Z\right) 
\simeq%
\underset{x\in V}{colim}~\left( Hom_{\mathbb{AB}}\left( A,Z\right) \right)
^{V}%
\simeq
\\
&&%
\simeq%
\left( \kappa \left( pro\text{-}H_{0}\left( \_,A\right) \right) \left(
Z\right) \right) _{x}%
\simeq%
\left( \kappa \left( \mathcal{H}\right) \right) _{x}\left( Z\right) .
\end{eqnarray*}%
Therefore $\kappa \left( A\right) _{x}%
\simeq%
\left( \kappa \left( pro\text{-}H_{0}\left( \_,A\right) \right) \right) _{x}$
and $\left( pro\text{-}H_{0}\left( \_,A\right) \right) ^{x}%
\simeq%
\left( A\right) ^{x}$, as desired.
\end{proof}

\section{\label{Appendix-Pro-category}Appendix}

\subsection{Categories $Pro\left( \mathbb{SET}\right) $ and $Pro\left( 
\mathbb{AB}\right) $}

Let us remind necessary notions from category theory. We fix a \textbf{%
universe} $\mathfrak{U}$ (\cite{Kashiwara-MR2182076}, Definition 1.1.1).

\begin{definition}
\label{Small-set}A set is called \textbf{small} ($\mathfrak{U}$-small in the
terminology of \cite{Kashiwara-MR2182076}, Definition 1.1.2) if it is
isomorphic to a set belonging to $\mathfrak{U}$. A category $\mathbb{D}$ is
called \textbf{small} if both the set of objects $Ob\left( \mathbb{D}\right) 
$ and the set of morphisms $Mor\left( \mathbb{D}\right) $ are small.
\end{definition}

\begin{definition}
\label{U-category}A category $\mathbb{D}$ is called a $\mathfrak{U}$%
-category (\cite{Kashiwara-MR2182076}, Definition 1.2.1) if 
\begin{equation*}
Hom_{\mathbb{D}}\left( X,Y\right)
\end{equation*}%
is small for any two objects $X$ and $Y$.
\end{definition}

\begin{definition}
\label{Small limits}A \textbf{small limit} (\textbf{small colimit}) in a
category $\mathbb{D}$ is a limit (colimit) of a diagram%
\begin{equation*}
X:I\longrightarrow \mathbb{D}
\end{equation*}%
where $I$ is a small category.
\end{definition}

\begin{definition}
\label{Filtrant-colimit}A \textbf{filtrant (cofiltrant) colimit} (\textbf{%
limit}) in a category $\mathbb{D}$ is a colimit (limit) of a diagram%
\begin{equation*}
X:I\longrightarrow \mathbb{D}
\end{equation*}%
where $I$ is a small filtrant (cofiltrant) category (\cite%
{Kashiwara-MR2182076}, Definition 3.1.1).
\end{definition}

\begin{remark}
Whenever possible, we use simplified notations for diagrams: $\left(
X_{i}\right) _{i\in I}$, their limits:%
\begin{equation*}
\lim_{i}X_{i}=\left( Y\longrightarrow X_{i}\right) _{i\in I},
\end{equation*}%
and colimits 
\begin{equation*}
\underset{i}{colim}X_{i}=\left( X_{i}\longrightarrow Y\right) _{i\in I}.
\end{equation*}
\end{remark}

\begin{remark}
\label{Empty-diagram}For an empty diagram $X:I\rightarrow \mathbb{D}$, its
colimit is an initial object in $\mathbb{D}$, while its limit a terminal
object. In particular, a coproduct of an empty family of objects is an
initial object, while a product of such a family is terminal.
\end{remark}

\begin{definition}
\label{complete}A category $\mathbb{D}$ is called \textbf{complete} (\textbf{%
cocomplete}) if $\mathbb{D}$ admits small limits (colimits).
\end{definition}

\begin{definition}
\label{SET,AB}We denote by $\mathbb{SET}$ the category of small sets, and by 
$\mathbb{AB}$ the additive category of small abelian groups. These two
categories are clearly $\mathfrak{U}$-categories.
\end{definition}

\begin{definition}
\label{setset,abab}For a category $\mathbb{D}$, let $\mathbb{SET}^{\mathbb{D}%
}$ be the category of functors $\mathbb{D\rightarrow SET}$. For an \textbf{%
additive} category $\mathbb{D}$, let $\mathbb{AB}^{\mathbb{D}}$ be the
category of \textbf{additive} functors $\mathbb{D\rightarrow AB}$. These two
categories are \textbf{not} in general $\mathfrak{U}$-categories (unless $%
\mathbb{D}$ is a small category).
\end{definition}

For a category $\mathbb{D}$, let $\iota :\mathbb{D}^{op}\rightarrow \mathbb{%
SET}^{\mathbb{D}}$ be the Ioneda full embedding:%
\begin{equation*}
\iota \left( X\right) :=Hom_{\mathbb{D}}\left( X,\_\right) :\mathbb{D}%
\longrightarrow \mathbb{SET},
\end{equation*}%
and let $\kappa :\mathbb{D}\longrightarrow \mathbb{SET}^{\mathbb{D}}$ be the
corresponding contravariant embedding.

\begin{definition}
\label{Def-proset}(\cite{Kashiwara-MR2182076}, Definition 6.1.1) The
category $Pro\left( \mathbb{D}\right) $ is the opposite category $\left( 
\mathbb{E}\right) ^{op}$ where $\mathbb{E}\subseteq \mathbb{SET}^{\mathbb{D}%
} $ is the full subcategory of functors that are filtrant colimits of
representable functors, i.e. colimits of diagrams of the form%
\begin{equation*}
I^{op}\overset{X^{op}}{\longrightarrow }\mathbb{D}^{op}\overset{\iota }{%
\longrightarrow }\mathbb{SET}^{\mathbb{D}}
\end{equation*}%
where $I$ is a small cofiltrant category ($I^{op}$ is filtrant), and $%
X:I\rightarrow \mathbb{D}$ is a functor.

Let two pro-objects be defined by the diagrams $\left( X_{i}\right) _{i\in
I} $ and $\left( Y_{j}\right) _{j\in J}$. Then%
\begin{equation*}
Hom_{Pro\left( \mathbb{D}\right) }\left( \left( X_{i}\right) _{i\in
I},\left( Y_{j}\right) _{j\in J}\right) =\lim_{j}~\underset{i}{colim}~Hom_{%
\mathbb{D}}\left( X_{i},Y_{j}\right) .
\end{equation*}
\end{definition}

\begin{remark}
\label{Trivial-pro-object}The category $\mathbb{D}$ is fully embedded in $%
Pro\left( \mathbb{D}\right) $: an object $X\in \mathbb{D}$ can be
represented by a trivial diagram $\left( X_{i}=X\right) _{i\in I}$ where $I$
is a category with one object and one morphism. Such pro-objects are called 
\emph{trivial} (\emph{rudimentary} in \cite{Mardesic-Segal-MR676973})
\end{remark}

\begin{definition}
\label{i-SET}The full embedding 
\begin{equation*}
\iota :\left( Pro\left( \mathbb{D}\right) \right) ^{op}\longrightarrow 
\mathbb{SET}^{\mathbb{D}}
\end{equation*}%
will be also called the Ioneda embedding, and will be denoted by the same
symbol $\iota $.
\end{definition}

Let $\mathbb{D}$ be an additive category. Since any set $\iota \left(
X\right) \left( Y\right) =Hom_{\mathbb{D}}\left( X,Y\right) $ has the
structure of an abelian group, and the bifunctors $Hom_{\mathbb{D}}\left(
\_,\_\right) $ are bi-additive, we obtain the \textbf{additive} Ioneda
embedding, denoted by the same letter%
\begin{equation*}
\iota :\left( Pro\left( \mathbb{D}\right) \right) ^{op}\longrightarrow 
\mathbb{AB}^{\mathbb{D}}.
\end{equation*}

\begin{remark}
\label{Belongs-to-pro-D}An object $Y\in \mathbb{SET}^{\mathbb{D}}$ ($\mathbb{%
AB}^{\mathbb{D}}$) belongs to $\iota \left( \left( Pro\left( \mathbb{D}%
\right) \right) ^{op}\right) $ iff it is right exact and its fiber category $%
C^{Y}$ (\cite{Kashiwara-MR2182076}, Definition 1.2.16) is finally small. See 
\cite{Kashiwara-MR2182076}, Propositions 6.1.5 and 8.6.2.
\end{remark}

\begin{remark}
\label{Belongs-to-D}Let $\mathbb{D}$ be either $\mathbb{SET}$ or $\mathbb{AB}
$. The restrictions of $\iota $ on $\mathbb{D}^{op}$ are the classical
Ioneda (general and additive) embedding%
\begin{equation*}
\iota :\mathbb{D}^{op}\longrightarrow \mathbb{D}^{\mathbb{D}}.
\end{equation*}%
An object $Y\in \mathbb{D}^{\mathbb{D}}$ belongs to $\iota \left( \mathbb{D}%
^{op}\right) $ iff $C^{Y}$ has an initial object (\cite{Kashiwara-MR2182076}%
, Proposition 1.4.10).
\end{remark}

\begin{definition}
\label{k-SET}Let us denote by $\kappa $ the corresponding contravariant
embedding%
\begin{equation*}
\kappa :Pro\left( \mathbb{D}\right) \longrightarrow \mathbb{SET}^{\mathbb{D}%
}\left( \mathbb{AB}^{\mathbb{D}}\right) .
\end{equation*}
\end{definition}

Proposition \ref{k-SET-Remark} below is valid for general categories of the
form $\mathbb{SET}^{\mathbb{D}}$ or $\mathbb{AB}^{\mathbb{D}}$. We will use
it only in the cases $\mathbb{D=SET}$ or $\mathbb{D=AB}$:

\begin{proposition}
\label{k-SET-Remark}Let $\mathbb{D}$ be either $\mathbb{SET}$ or $\mathbb{AB}
$.

\begin{enumerate}
\item \label{DD-vs-D-colimits}Morphisms $\left( X_{i}\rightarrow Y\right)
_{i\in I}$, where $I$ is a small category, form a colimit in $\mathbb{D}^{%
\mathbb{D}}$ iff $\left( X_{i}\left( Z\right) \rightarrow Y\left( Z\right)
\right) _{i\in I}$ form a colimit in $\mathbb{D}$ for any $Z\in \mathbb{D}$.

\item \label{DD-vs-D-limits}Morphisms $\left( Y\rightarrow X_{i}\right)
_{i\in I}$, where $I$ is a small category, form a limit in $\mathbb{D}^{%
\mathbb{D}}$ iff $\left( Y\left( Z\right) \rightarrow X_{i}\left( Z\right)
\right) _{i\in I}$ form a limit in $\mathbb{D}$ for any $Z\in \mathbb{D}$.

\item \label{DD-vs-D-mono-epi}A morphism $f:X\rightarrow Y$ in $\mathbb{D}^{%
\mathbb{D}}$ is a monomorphism (epimorphism) iff $f\left( Z\right) :X\left(
Z\right) \rightarrow Y\left( Z\right) $ is a monomorphism (epimorphism) in $%
\mathbb{D}$ for any $Z\in \mathbb{D}$.

\item \label{DD-vs-Pro(D)-colimits}The \textbf{contravariant} embedding%
\begin{equation*}
\kappa :Pro\left( \mathbb{D}\right) \longrightarrow \mathbb{D}^{\mathbb{D}}
\end{equation*}%
converts small colimits in $Pro\left( \mathbb{D}\right) $ to limits in $%
\mathbb{D}^{\mathbb{D}}$. Moreover, morphisms $\left( X_{i}\rightarrow
Y\right) _{i\in I}$, where $I$ is a small category, form a colimit in $%
Pro\left( \mathbb{D}\right) $ iff $\left( \kappa \left( Y\right) \rightarrow
\kappa \left( X_{i}\right) \right) _{i\in I}$ form a limit in $\mathbb{D}^{%
\mathbb{D}}$.

\item \label{DD-vs-Pro(D)-limits}The embedding $\kappa $ converts cofiltrant
limits in $Pro\left( \mathbb{D}\right) $ to filtrant colimits in $\mathbb{D}%
^{\mathbb{D}}$. Moreover, given a small cofiltrant diagram $\left(
X_{i}\right) _{i\in I}$, then the morphisms $\left( Y\rightarrow
X_{i}\right) _{i\in I}$ form a limit in $Pro\left( \mathbb{D}\right) $ iff $%
\left( \kappa \left( X_{i}\right) \rightarrow \kappa \left( X\right) \right)
_{i\in I}$ form a colimit in $\mathbb{D}^{\mathbb{D}}$.

\item \label{DD-vs-Pro(D)-epi}A morphism $f:X\rightarrow Y$ in $Pro\left( 
\mathbb{D}\right) $ is an epimorphism iff $\kappa \left( f\right) $ is a
monomorphism in $\mathbb{D}^{\mathbb{D}}$.

\item \label{Cofiltrant-limits-exact}Cofiltrant limits are exact in $%
Pro\left( \mathbb{D}\right) $.

\item \label{Pro(D)-vs-D-limits-colimits}The full embedding $\mathbb{D}%
\hookrightarrow Pro\left( \mathbb{D}\right) $ commutes with \textbf{small}
colimits (e.g., cokernels and \textbf{small} coproducts), and with \textbf{%
finite} limits (e.g., kernels and \textbf{finite} products),

\item \label{Pro(D)-vs-D-mono-epi}The full embedding $\mathbb{D}%
\hookrightarrow Pro\left( \mathbb{D}\right) $ preserves epimorphisms and
monomorphisms, i.e. $f:X\rightarrow Y$ is an epimorphism (monomorphism) in $%
\mathbb{D}$ iff $f$ is an epimorphism (monomorphism) in $Pro\left( \mathbb{D}%
\right) $.
\end{enumerate}
\end{proposition}

\begin{proof}
These statements are more or less well-known. For the proof, see \cite%
{Kashiwara-MR2182076}, Part 6 and Chapter 8.6.
\end{proof}

\subsection{Grothendieck topologies (sites)}

\begin{definition}
\label{Grothendieck-site}A \textbf{Grothendieck topology}, or a \textbf{%
Grothendieck site} (or simply a \textbf{site}) is a pair%
\begin{equation*}
X=\left( Cat\left( X\right) ,Cov\left( X\right) \right)
\end{equation*}%
where $Cat\left( X\right) $ is a small category (Definition \ref{Small-set}%
), and $Cov\left( X\right) $ is a collection of families of morphisms%
\begin{equation*}
\left\{ U_{i}\longrightarrow U\right\} \in Cat\left( X\right)
\end{equation*}%
satisfying (1)-(3) from \cite{Artin-GT}, Definition 1.1.1, COV1-COV4 from 
\cite{Kashiwara-MR2182076}, p. 391, or T1-T3 from \cite{Tamme-MR1317816},
Definition I.1.2.1.
\end{definition}

\begin{example}
\label{Site-TOP}Let $X$ be a topological space. We will call the site $%
OPEN\left( X\right) $ below the \textbf{standard site} for $X$:%
\begin{equation*}
OPEN\left( X\right) =\left( Cat\left( OPEN\left( X\right) \right) ,Cov\left(
OPEN\left( X\right) \right) \right) .
\end{equation*}%
$Cat\left( OPEN\left( X\right) \right) $ will consist of open subsets of $X$
as objects and inclusions $U\subseteq V$ as morphisms. The set of coverings $%
Cov\left( OPEN\left( X\right) \right) $ consists of families%
\begin{equation*}
\left\{ U_{i}\longrightarrow U\right\} \in Cat\left( OPEN\left( X\right)
\right)
\end{equation*}%
with%
\begin{equation*}
\dbigcup\limits_{i}U_{i}=U.
\end{equation*}
\end{example}

\begin{remark}
\label{Denote-standard-site-simply}We will often denote the standard site
simply by $X=\left( Cat\left( X\right) ,Cov\left( X\right) \right) $.
\end{remark}

\begin{example}
\label{Site-NORM}Let again $X$ be a topological space. Consider the site%
\begin{equation*}
NORM\left( X\right) =\left( Cat\left( NORM\left( X\right) \right) ,Cov\left(
NORM\left( X\right) \right) \right)
\end{equation*}%
where%
\begin{equation*}
Cat\left( NORM\left( X\right) \right) =Cat\left( X\right)
\end{equation*}%
while $Cov\left( NORM\left( X\right) \right) $ consists of \textbf{normal}
coverings%
\begin{equation*}
\left\{ U_{i}\longrightarrow U\right\} \in Cat\left( X\right) .
\end{equation*}%
A normal covering (\cite{Mardesic-Segal-MR676973}, I, \S 6.2) is a covering $%
\left\{ U_{i}\right\} $ which admits a partition of unity subordinated to $%
\left\{ U_{i}\right\} $.
\end{example}

\begin{example}
\label{Site-FINITE}Let again $X$ be a topological space. Consider the site%
\begin{equation*}
FINITE\left( X\right) =\left( Cat\left( FINITE\left( X\right) \right)
,Cov\left( FINITE\left( X\right) \right) \right)
\end{equation*}%
where%
\begin{equation*}
Cat\left( FINITE\left( X\right) \right) =Cat\left( X\right)
\end{equation*}%
while $Cov\left( FINITE\left( X\right) \right) $ consists of \textbf{finite}
normal coverings%
\begin{equation*}
\left\{ U_{i}\longrightarrow U\right\} \in Cat\left( X\right) .
\end{equation*}
\end{example}

\begin{example}
\label{Equivariant}Let $G$ be a topological group, and $X$ be a $G$-space.
The corresponding site $OPEN_{G}\left( X\right) $ has $G$-invariant open
subsets of $X$ as objects of $Cat\left( OPEN_{G}\left( X\right) \right) $
and $G$-invariant open coverings as elements of $Cov\left( OPEN_{G}\left(
X\right) \right) $ (compare to \cite{Artin-GT}, Example 1.1.4, or \cite%
{Tamme-MR1317816}, Example (1.3.2)).
\end{example}

\begin{example}
\label{Site-ETALE}Let $X$ be a noetherian scheme, and define the site $%
X^{et} $ by: $Cat\left( X^{et}\right) $ is the category of schemes $Y/X$ 
\'{e}tale, finite type, while $Cov\left( X^{et}\right) $ consists of finite
surjective families of maps. See \cite{Artin-GT}, Example 1.1.6, or \cite%
{Tamme-MR1317816}, II.1.2.
\end{example}

Let $\mathbb{D}$ be a category.

\begin{definition}
\label{Presheaf in D}Let $X=\left( Cat\left( X\right) ,Cov\left( X\right)
\right) $ be a site. A \textbf{presheaf} on $X$ with values in $\mathbb{D}$
is a \textbf{contravariant} functor $\mathcal{A}$ from $Cat\left( X\right) $
to $\mathbb{D}$. \textbf{Morphisms} between presheaves are morphisms between
the corresponding functors.
\end{definition}

Assume $\mathbb{D}$ admits small products.

\begin{definition}
\label{Separated-presheaf}A presheaf $\mathcal{A}$ on a site $X=\left(
Cat\left( X\right) ,Cov\left( X\right) \right) $ with values in $\mathbb{D}$
is called \textbf{separated} (\textbf{monopresheaf} in the terminology of 
\cite{Bredon-Book-MR1481706}) iff%
\begin{equation*}
\mathcal{A}\left( U\right) \longrightarrow \dprod\limits_{i}\mathcal{A}%
\left( U_{i}\right)
\end{equation*}%
is a monomorphism for any covering $\left\{ U_{i}\longrightarrow U\right\}
\in Cov\left( X\right) $.
\end{definition}

\begin{remark}
The condition in the Definition \ref{Separated-presheaf} is equivalent to
the condition (S1) in \cite{Bredon-Book-MR1481706}, I.1.7.
\end{remark}

Assume $\mathbb{D}$ is \textbf{complete} (Definition \ref{complete}).

\begin{definition}
\label{Sheaf in D}A \textbf{sheaf} on a site $X=\left( Cat\left( X\right)
,Cov\left( X\right) \right) $ with values in $\mathbb{D}$ is a presheaf $%
\mathcal{A}$ such that%
\begin{equation*}
\mathcal{A}\left( U\right) 
\simeq%
\ker \left( \dprod\limits_{i}\mathcal{A}\left( U_{i}\right)
\rightrightarrows \dprod\limits_{i,j}\mathcal{A}\left( U_{i}\times
_{U}U_{j}\right) \right)
\end{equation*}%
for any covering $\left\{ U_{i}\longrightarrow U\right\} \in Cov\left(
X\right) $.
\end{definition}

\begin{remark}
The condition in the Definition \ref{Separated-presheaf} is equivalent to
the conditions (S1) and (S2) in \cite{Bredon-Book-MR1481706}, I.1.7. Those
conditions, in turn, are equivalent to the \textquotedblleft
classical\textquotedblright\ definition of a sheaf as a local homeomorphism $%
\mathcal{A}\rightarrow X$. We do not use the classical definition in this
paper. The reasons are:

\begin{enumerate}
\item The classical definition cannot be applied to sheaves with values in
more general categories like $\mathbb{SET}^{\mathbb{SET}}$ and $\mathbb{AB}^{%
\mathbb{AB}}$.

\item It cannot be translated into the language of cosheaves.

\item Many of our results are valid not only for the site $OPEN\left(
X\right) $ where $X$ is a topological space, but also for more general sites
like $NORM\left( X\right) $, $FINITE\left( X\right) $, the equivariant site
(see Examples \ref{Site-TOP}, \ref{Site-NORM}, \ref{Site-FINITE} and \ref%
{Equivariant}), and even for \textquotedblleft
non-topological\textquotedblright\ sites like the site $X^{et}$ (Example \ref%
{Site-ETALE}).
\end{enumerate}
\end{remark}

\begin{definition}
\label{Categories-of-(pre)sheaves}Let us denote:

\textbf{a)} by $\mathbb{PS}\left( X,\mathbb{D}\right) $ the category of
presheaves on $X$ with values in $\mathbb{D}$;

\textbf{b)} ($\mathbb{D}$ is complete) by $\mathbb{S}\left( X,\mathbb{D}%
\right) $ the full subcategory of $\mathbb{PS}\left( X,\mathbb{D}\right) $
consisting of sheaves.
\end{definition}

\subsection{\label{Sheafification}Sheafification}

Throughout Subsections \ref{Sheafification} and \ref%
{Sheaves-presheaves-on-topological-spaces}, $\mathbb{D}$ will denote $%
\mathbb{SET}$, $\mathbb{AB}$, $\mathbb{SET}^{\mathbb{SET}}$, or $\mathbb{AB}%
^{\mathbb{AB}}$.

\begin{definition}
\label{H0-sheaves}Let $\mathcal{A}$ be a presheaf with values in $\mathbb{D}$%
. For $\left\{ U_{i}\rightarrow U\right\} \in Cov\left( X\right) $, define%
\begin{equation*}
H^{0}\left( \left\{ U_{i}\longrightarrow U\right\} ,\mathcal{A}\right)
:=\ker \left( \dprod\limits_{i}\mathcal{A}\left( U_{i}\right)
\rightrightarrows \dprod\limits_{i,j}\mathcal{A}\left( U_{i}\cap
U_{j}\right) \right) .
\end{equation*}
\end{definition}

\begin{definition}
\label{Refinement}Given two coverings $\mathcal{V},\mathcal{U}\in Cov\left(
X\right) $,%
\begin{eqnarray*}
\mathcal{U} &=&\left\{ U_{i}\longrightarrow U\right\} _{i\in I}, \\
\mathcal{V} &=&\left\{ V_{j}\longrightarrow U\right\} _{j\in J},
\end{eqnarray*}%
then a \textbf{refinement mapping} $f:\mathcal{V}\longrightarrow \mathcal{U}$
is a pair%
\begin{equation*}
\left( \varepsilon :J\longrightarrow I,\left( f_{j}:V_{j}\longrightarrow
U_{\varepsilon \left( j\right) }\right) \right) _{j\in J},
\end{equation*}%
where $f_{j}$ are $U$-morphisms.
\end{definition}

\begin{lemma}
\label{Two-refinements}Given two coverings $\mathcal{V},\mathcal{U}\in
Cov\left( X\right) $, and two refinement mappings $f,g:\mathcal{V}%
\longrightarrow \mathcal{U}$, the corresponding mappings of kernels coincide:%
\begin{equation*}
H^{0}\left( f,\mathcal{A}\right) =H^{0}\left( g,\mathcal{A}\right)
:H^{0}\left( \mathcal{U},\mathcal{A}\right) \longrightarrow H^{0}\left( 
\mathcal{V},\mathcal{A}\right) .
\end{equation*}
\end{lemma}

\begin{proof}
For $\mathbb{D}$ being $\mathbb{SET}$ or $\mathbb{AB}$, see \cite%
{Tamme-MR1317816}, Lemma I.2.2.7. If $\mathbb{D}=\mathbb{SET}^{\mathbb{SET}}$
or $\mathbb{AB}^{\mathbb{AB}}$, let $Z$ be an arbitrary set (abelian group).
Apply \emph{loc. cit.} to the morphisms $f\left( Z\right) $ and $g\left(
Z\right) $, and get%
\begin{equation*}
H^{0}\left( f,\mathcal{A}\left( Z\right) \right) =H^{0}\left( g,\mathcal{A}%
\left( Z\right) \right) :H^{0}\left( \mathcal{U},\mathcal{A}\left( Z\right)
\right) \longrightarrow H^{0}\left( \mathcal{V},\mathcal{A}\left( Z\right)
\right) .
\end{equation*}
\end{proof}

\begin{remark}
In \cite{Tamme-MR1317816} all the reasonings are done for presheaves of 
\textbf{abelian groups}. However, as is underlined in (\cite{Tamme-MR1317816}%
, Remark (3.1.5)), the proofs can be easily translated to the situation of
presheaves of \textbf{sets} (see also \cite{SGA4-2-MR0354653}, exp. II, 6.4).
\end{remark}

\begin{definition}
Given $U\in Cat\left( X\right) $, the set of coverings on $U$ is a
cofiltrant pre-ordered set under the refinement relation: $\mathcal{V}\leq 
\mathcal{U}$ iff $\mathcal{V}$ refines $\mathcal{U}$. Since the mappings $%
H^{0}\left( \mathcal{U},\mathcal{A}\right) \longrightarrow H^{0}\left( 
\mathcal{V},\mathcal{A}\right) $ do not depend on the refinement mapping
(Lemma \ref{Two-refinements}), one can define%
\begin{equation*}
\mathcal{A}^{+}\left( U\right) :=\underset{\mathcal{V}}{colim}~H^{0}\left( 
\mathcal{V},\mathcal{A}\right)
\end{equation*}%
where $\mathcal{V}$ runs over coverings on $U$. $\mathcal{A}^{+}$ is clearly
a presheaf with values in $\mathbb{AB}$.
\end{definition}

\begin{proposition}
\label{Properties-Plus-construction}Let $\mathcal{A}$ be a sheaf with values
in $\mathbb{D}$. Then

\begin{enumerate}
\item $\mathcal{A}^{+}$ is separated.

\item If $\mathcal{A}$ is separated, then $\mathcal{A}^{+}$ is a sheaf.

\item The functor%
\begin{equation*}
\left( {}\right) ^{\#}:=\left( {}\right) ^{++}:\mathbb{PS}\left( X,\mathbb{D}%
\right) \longrightarrow \mathbb{S}\left( X,\mathbb{D}\right)
\end{equation*}%
is left adjoint to the inclusion functor.
\end{enumerate}
\end{proposition}

\begin{proof}
For $\mathbb{D}$ being $\mathbb{SET}$ or $\mathbb{AB}$, see \cite%
{Tamme-MR1317816}, Proposition I.3.1.3. If $\mathbb{D}=\mathbb{SET}^{\mathbb{%
SET}}$ or $\mathbb{AB}^{\mathbb{AB}}$, let $Z$ be an arbitrary set (abelian
group). Apply \emph{loc. cit.} to the presheaf of sets (abelian groups) $%
\mathcal{A}\left( Z\right) $. Varying $Z$, one gets the desired adjunction.
\end{proof}

\subsection{\label{Sheaves-presheaves-on-topological-spaces}Sheaves and
presheaves on topological spaces}

Throughout this Subsection, $X$ is a topological space considered as a site $%
OPEN\left( X\right) $ (see Example \ref{Site-TOP} and Remark \ref%
{Denote-standard-site-simply}).

\begin{definition}
\label{Stalk}Let $\mathcal{A}$ be a presheaf on $X$ with values in $\mathbb{D%
}$. Denote by%
\begin{equation*}
\mathcal{A}_{x}%
{:=}%
\underset{x\in U}{colim}~\mathcal{A}\left( U\right) .
\end{equation*}%
We will call $\mathcal{A}_{x}$ the \textbf{stalk} of $\mathcal{A}$ at $x$.
\end{definition}

\begin{remark}
It follows from Proposition \ref{k-SET-Remark} (\ref{DD-vs-D-limits}) that
if $\mathcal{A}$ is a presheaf with values in $\mathbb{SET}^{\mathbb{SET}}$
or $\mathbb{AB}^{\mathbb{AB}}$, and if $Z\in \mathbb{SET}$ ($Z\in \mathbb{AB}
$) then $\mathcal{A}_{x}\left( Z\right) $ is canonically isomorphic to $%
\left( \mathcal{A}\left( Z\right) \right) _{x}$.
\end{remark}

\begin{definition}
\label{local-isomorphism-presheaves}Let $\mathcal{A\rightarrow B}$ be a
morphism of presheaves on $X$. It is called a \textbf{local isomorphism} iff 
$\mathcal{A}_{x}\rightarrow \mathcal{B}_{x}$ is an isomorphism for any $x\in
X$.
\end{definition}

\begin{proposition}
\label{Sheaves-OPEN(X)-DD}~

\begin{enumerate}
\item Let $\varphi :\mathcal{A\rightarrow B}$ be a local isomorphism of
sheaves with values in $\mathbb{D}$ on $X$. Then $\varphi $ is an
isomorphism.

\item Given a presheaf $\mathcal{B}$ with values in $\mathbb{D}$ on $X$, the
natural morphisms%
\begin{eqnarray*}
\beta ^{+}\left( \mathcal{B}\right) &:&\mathcal{B}\longrightarrow \mathcal{B}%
^{+}, \\
\beta ^{\#}\left( \mathcal{B}\right) &=&\beta ^{+}\left( \mathcal{B}%
^{+}\right) \circ \beta ^{+}\left( \mathcal{B}\right) :\mathcal{B}%
\longrightarrow \mathcal{B}^{\#},
\end{eqnarray*}%
are local isomorphisms.
\end{enumerate}
\end{proposition}

\begin{proof}
For $\mathbb{D}$ being $\mathbb{SET}$ or $\mathbb{AB}$, the statements above
are more or less well-known. If $\mathbb{D}=\mathbb{SET}^{\mathbb{SET}}$ or $%
\mathbb{AB}^{\mathbb{AB}}$, consider an arbitrary set (abelian group) $Z$,
and apply the Proposition to the (pre)sheaves $\mathcal{A}\left( Z\right) $
and $\mathcal{B}\left( Z\right) $.
\end{proof}

\bibliographystyle{alpha}
\bibliography{Cosheaves}

\end{document}